\documentclass[a4paper,11pt]{amsart}

\usepackage{amssymb,amsmath,amsthm,mathrsfs,enumerate,graphicx, color}
\usepackage[pdfpagelabels,colorlinks,linkcolor=blue,citecolor=black,urlcolor=blue]{hyperref}
\usepackage{esint}
\newtheorem{thm}{Theorem}[section]

\newtheorem{lem}[thm]{Lemma}
\newtheorem{prop}[thm]{Proposition}
\newtheorem{defn}[thm]{Definition}
\newtheorem{rem}[thm]{Remark}

\newcommand{\lesi}{\lesssim}

\newcommand{\dx}{d\mu(x)}
\newcommand{\dy}{d\mu(y)}

\newcommand{\supp}{\operatorname{supp}}

\newcommand{\f}{\frac}

\newcommand{\Om}{\Omega}

\newcommand{\su}{\subset}
\newcommand{\vc}{\infty}
\newcommand{\Rn}{\mathbb{R}^n}

\newcommand{\rad}{\rm rad}

\title[Weighted Hardy space estimates on Green operators ]{Regularity estimates for Green operators of Dirichlet and Neumann problems on weighted Hardy spaces}         
\author{The Anh Bui}
\address{Department of Mathematics, Macquarie University, NSW 2109,
Australia}
\email{the.bui@mq.edu.au, bt\_anh80@yahoo.com}
 
\author{Xuan Thinh Duong}
\address{Department of Mathematics, Macquarie University, NSW 2109,
Australia}

\email{xuan.duong@mq.edu.au}


\begin{document}

\begin{abstract}
In this paper we first study the generalized weighted Hardy spaces $H^p_{L,w}(X)$ for $0<p\le 1$ associated to nonnegative self-adjoint operators $L$ satisfying Gaussian upper bounds on the space of homogeneous type $X$ in both cases of finite and infinite measure. We show that the weighted Hardy spaces defined via  maximal functions and atomic decompositions coincide. Then we  prove weighted regularity estimates for the Green operators of the inhomogeneous Dirichlet and Neumann problems in suitable bounded or unbounded domains including  bounded semiconvex domains,  convex regions above a Lipschitz graph and upper half-spaces. Our estimates are in terms of weighted $L^p$ spaces for the range $1<p<\vc$ and in terms of the new weighted Hardy spaces for the range $0<p\le 1$. Our regularity estimates for the Green operators under the weak smoothness assumptions on the boundaries of the domains are new, especially the estimates on Hardy spaces for the full range $0<p\le 1$ and the case of unbounded domains.
\end{abstract}
\date{}

\maketitle

\tableofcontents

\section{Introduction}\label{sec: intro}
Let $\Om$ be an open connected domain in $\Rn$. Denote by $W^{1,2}(\Om)$ the  Sobolev space on $\Om$ with the norm 
\[
\|f\|_{W^{1,2}(\Om)}=\|f\|_{L^2(\Om)} + \|\nabla f\|_{L^2(\Om)}.
\]
The closure of $C^\vc_c(\Om)$ in $W^{1,2}(\Om)$  will be denoted by $W^{1,2}_0(\Om)$. 

Consider the inhomogeneous Dirichlet problem for the Laplacian
\begin{equation}\label{DirichletProblem}
\left\{\begin{array}{ccccc}
\Delta u &=& f &\text{in}& \Om\\
u &=& 0&\text{on}& \partial \Om.
\end{array}\right.
\end{equation}
Denote by $\mathbb{G}_D$ the Green operator for Dirichlet problem (\ref{DirichletProblem}), i.e. the solution operator which maps each $f\in C^\vc(\overline \Om)$ to the unique solution $u:=\mathbb{G}_D(f)\in W^{1,2}_0(\Om)$ of the problem \eqref{DirichletProblem}.

We also consider the inhomogeneous Neumann problem for the Laplacian
\begin{equation}\label{NeumannProblem}
\left\{\begin{array}{ccccc}
\Delta u &=& f &\text{in}& \Om\\
\partial_\nu u &=& 0&\text{on}& \partial \Om
\end{array}\right.
\end{equation} 
for a suitable domain $\Omega$. Denote by $\mathbb{G}_N$ the Green operator for Neumann problem \eqref{NeumannProblem}, i.e. the solution operator which maps each $f\in C^\vc(\overline \Om)$ with $\int_\Om f =0$ to the unique solution $u:=\mathbb{G}_N(f)\in W^{1,2}(\Om)$ of the problem \eqref{NeumannProblem}.

One of the most interesting and important problems concerning problems \eqref{DirichletProblem} and \eqref{NeumannProblem} is the regularity estimate for the Green operators $\mathbb{G}_D$ and $\mathbb{G}_N$. We would like to give a shortlist of known results in this research direction for the $L^p$-boundedness with $1<p<\vc$ (see for example \cite{DHMMY}).
\begin{enumerate}[{\rm (i)}]
	\item The $L^p$-boundedness for $\nabla^2\mathbb{G}_D$ and $\nabla^2 \mathbb{G}_N$ with $1<p<\vc$  on a bounded $C^\vc$ domain was obtained in \cite{ADN} and \cite{LM}. 
	\item In  \cite{Ka}, it was proved that $\nabla^2\mathbb{G}_D$ is well-defined and bounded on $L^2(\Om)$ provided that $\Om$ is a bounded convex domain.
	\item Under the assumption that $\Om$ is a bounded and convex domain, the weak type (1,1) for  $\nabla^2\mathbb{G}_D$ was proved in \cite{DVW, F}, meanwhile  the boundedness for  $\nabla^2\mathbb{G}_D$ on the suitable Hardy space was obtained in \cite{A2}.
	\item The $L^2$-boundedness for  $\nabla^2\mathbb{G}_N$ appeared first in \cite{GI}. Then it was proved that it is bounded from some Hardy space into $L^1(\Om)$, hence by interpolation it is bounded on $L^p(\Om)$ for $1<p<2$. See for example \cite{AJ}.
	\item It is important to note that the $L^p$--boundedness for $\nabla^2\mathbb{G}_D$ and $\nabla^2 \mathbb{G}_N$ may fail in the class of Lipschitz domains for any $p\in (1,\vc)$ and in the class of convex domains for any $p\in (2,\vc)$. For the further details, see \cite{A2,AJ, D, JK, MP} and the references therein.
\end{enumerate}
The following brief summary gives an overview of the progress concerning the boundedness for $\nabla^2\mathbb{G}_D$ and $\nabla^2 \mathbb{G}_N$ for $0<p\le 1$.
\begin{enumerate}[{\rm (i)}]
	\item In  \cite{CDS, CKS1, CKS2} the authors studied the theory of Hardy spaces on domains. In 
	\cite{CDS}, they obtained the boundedness of  $\nabla^2\mathbb{G}_D$ and $\nabla^2 \mathbb{G}_N$  on these Hardy spaces with the range $0<p\le 1$  when the domains are  bounded $C^\vc$ domains. The boundedness on the Hardy spaces with the range $\f{n}{n+1}<p\le 1$ for $\nabla^2\mathbb{G}_D$ and $\nabla^2 \mathbb{G}_N$ on either bounded Lipschitz domains or the upper half-spaces was proved in  \cite{CKS1,CKS2} (see also \cite{MM}).
	\item In the case when $\Om$ is a bounded Lipschitz domain satisfying a uniform exterior ball condition,  the estimates of $\nabla^2\mathbb{G}_D$ on Besov and Triebel--Lizorkin spaces were proved in  \cite{MMY}. These results include the boundedness of $\nabla^2\mathbb{G}_D$ on local Hardy spaces for $\f{n}{n+1}<p\le 1$.
	\item Recently, the authors in \cite{DHMMY} developed the theory of Hardy spaces associated to Dirichlet Laplacians on bounded semiconvex domains and Neumann Laplacians on bounded convex domains. Then they gave a new approach to obtain the boundedness of  $\nabla^2\mathbb{G}_D$  and $\nabla^2\mathbb{G}_N$ on local Hardy spaces for $\f{n}{n+1}<p\le 1$. These results were extended to weighted Orlicz-Hardy spaces in \cite{CCYY}.
\end{enumerate}
Although the regularity  estimates for $\nabla^2\mathbb{G}_D$  and $\nabla^2\mathbb{G}_N$ have been investigated intensively, there are still a number of interesting open problems. 
 
\noindent \textbf{Problem 1:} The Hardy space estimates for $\nabla^2\mathbb{G}_D$  and $\nabla^2\mathbb{G}_N$ for a full range $0<p\le 1$ are only known when $\Om$ is a bounded $C^\vc$ domain. Under weaker smoothness assumptions such as Lipschitz domains, the range $\f{n}{n+1}<p\le 1$ is known but  the range $0<p\le \f{n}{n+1}$ is still open.

\noindent \textbf{Problem 2:} There are a number of results for the Hardy space estimates for $\nabla^2\mathbb{G}_D$  and $\nabla^2\mathbb{G}_N$ on bounded domains, while to the best of our knowledge similar results on unbounded domains are still open. See \cite{CKS2} for the boundedness of $\nabla^2\mathbb{G}_D$  and $\nabla^2\mathbb{G}_N$ on Hardy spaces with $\f{n}{n+1}<p\le 1$ on the upper half-spaces. See \cite{A1, AJ, GI} for the boundedness of $\nabla^2\mathbb{G}_D$  and $\nabla^2\mathbb{G}_N$ from the Hardy spaces $(p=1)$ to $L^1$ on certain unbounded Lipschitz domains. In the case of general unbounded domains, the Hardy space estimates for both  $\nabla^2\mathbb{G}_D$  and $\nabla^2\mathbb{G}_N$  are still unknown even for $p=1$.

\noindent \textbf{Problem 3:} Concerning the weighted estimates, recently in \cite{CCYY} the authors introduced the local weighted Orlicz-Hardy spaces in a bounded semiconvex/convex domain and they obtained  the boundedness of  $\nabla^2\mathbb{G}_D$  and $\nabla^2\mathbb{G}_N$  on these spaces with a limited range of $p$. However, it seems that the class of weights in  \cite{CCYY} is not optimal. It is natural to raise the question on finding better class of weights  and weighted $L^p$ estimates for $1<p<\vc$ not only for bounded domains but also for unbounded domains. 

The  aim of this paper is to address Problems 1, 2 and 3 for different types of domains. Our main results are for bounded domains in Theorem
\ref{mainthm1-bounded domain} and for unbounded domains in Theorem \ref{mainthm2-above convex Lipschitz domains}. In the specific case of upper half spaces,  in addition to estimates from Theorem \ref{mainthm2-above convex Lipschitz domains}, we give further results in 
Theorem \ref{mainthm3-upper half-space}. Our approach in this paper might be applicable to other problems in different settings since we state our assumptions on certain heat kernel estimates; see the theorems in Sections 3.1 and 3.2.

In order to state the main results precisely, we first give definitions of  weighted Hardy spaces.
For the weighted Hardy spaces on $\Rn$, we recall the definitions in \cite{ST}. Assume that $p\in (0,1]$, the weight $w$ belongs to the Muckenhoupt class $A_\vc(\Rn)$ (see Section 2 for the Muckenhoupt weights)  and $q\in (q_w,\vc]$ where $q_w$ is defined in (\ref{qw-defn}). A bounded, measurable function $a$ is called a $(p,q,w)$-atom if 
\begin{enumerate}[\upshape (i)]
	\item $a$ is supported in a ball $B\subset \mathbb{R}^n$;
	\item $\|a\|_{L^q_w}\le w(B)^{1/q-1/p}$;
	\item $\displaystyle \int_{B} x^\alpha a(x)dx=0$ 
	for all multi--indices $\alpha$ with $|\alpha|\le \lfloor n(q_w/p-1)\rfloor$.
\end{enumerate}
The Hardy space $H^{p,q}_{w}(\Om)$ is defined as the set of all distributions $f\in \mathscr{S}'$ such that
\[
f=\sum_{j=1}^{\infty}\lambda_j a_j
\]
where $a_j$ are $(p,q,w)$-atoms and $\lambda_j$ are scalars with $\sum_{j=1}^{\infty}|\lambda_j|^p<\vc$.
We also set
\[
\|f\|^p_{H^{p,q}_{w}(\Rn)}=\inf\Big\{\sum_{j=1}^{\infty}|\lambda_j|^p: f=\sum_{j=1}^{\infty}\lambda_j a_j \Big\}
\]
where the infimum is taken over all such decompositions.

It is well known that for $p\in (0,1]$, $w\in A_\vc(\Rn)$ and $q\in (q_w,\vc]$ we have
\[
H^{p,q}_{w}(\Rn)\equiv H^{p,\vc}_{w}(\Rn).
\]
Hence, for any $p\in (0,1]$ and $w\in A_\vc(\Rn)$  we define $H^{p}_{w}(\Rn)$ as any space $H^{p,q}_{w}(\Rn)$ with $q\in (q_w,\vc]$.

We next recall the weighted Hardy spaces on domains of Miyachi \cite{Mi2}.
\begin{defn}  Let $\Om$ is an open set in $\Rn$. Let $p\in (0,1]$, $w\in A_\vc(\Rn)$ and $q\in (q_w,\vc]$. A bounded, measurable function $a$ from $\Om$ to $\mathbb{R}$ is called a $(p,q,w)_{Mi}$-atom if 
	\begin{enumerate}[\upshape (i)]
		\item $a$ is supported in a ball $B\subset \Om$;
		\item $\|a\|_{L^q_w(\Om)}\le w(B)^{1/q-1/p}$;
		\item either $2B\subset \Om$ and  $4B\cap \partial\Om\ne \emptyset$, or $4B\subset \Om$ and 
		$$\displaystyle \int_B x^\alpha a(x)dx=0$$ 
		for all multi--indices $\alpha$ with $|\alpha|\le \lfloor n(q_w/p-1)\rfloor$.
	\end{enumerate}
	The Hardy space $H^{p,q}_{Mi,w}(\Om)$ is defined as the set of all $f\in \mathscr{S}'$ such that
	\[
	f=\sum_{j=1}^{\infty} \lambda_j a_j
	\]
	where $a_j$ are $(p,q,w)_{Mi}$-atoms and $\lambda_j$ are scalars with $\sum_{j=1}^{\infty} |\lambda_j|^p<\vc$.
	We also set
	\[
	\|f\|^p_{H^{p,q}_{Mi,w}(\Om)}=\inf\Big\{\sum_{j=1}^{\infty}|\lambda_j|^p: f=\sum_{j=1}^{\infty}\lambda_j a_j \Big\}
	\]
	where the infimum is taken over all such decompositions.
\end{defn}
Let $\phi \in C^\vc_c(B(0,1))$  be a non--negative radial function such that $\int \phi(x)dx=1$. It was proved in \cite{Mi} that the Hardy spaces $H^{p,q}_{Mi,w}(\Om)$ can be characterized in terms of maximal functions of the form
\begin{equation}\label{eq-phi}
f_\Om^+(x)=\max_{0<t<\delta(x)/2}|\phi_t\ast f(x)|
\end{equation}
where $\delta(x)=d(x,\Om^c)$ and $\phi_t(x)=t^{-n}\phi(x/t)$. More precisely, we have the following theorem from \cite{Mi2}:
\begin{thm}\label{MiyachiTheorem}
	Let $p\in (0,1]$, $w\in A_\vc(\Rn)$ and $q\in (q_w,\vc]$. Then we have $f\in H^{p,q}_{Mi,w}(\Om)$ if and only if $f_\Om^+\in L^p_w(\Om)$; moreover,
	\[
	\|f\|_{H^{p,q}_{Mi,w}(\Om)}\sim \|f_\Om^+\|_{L^p_w(\Om)}.
	\]
\end{thm}
From Theorem \ref{MiyachiTheorem} for $w\in A_\vc(\Rn)$ and $p\in (0,1]$ we will write the Hardy spaces $H^p_{Mi,w}(\Om)$ for any space $H^{p,q}_{Mi,w}(\Om)$ with  $q\in (q_w,\vc]$.

The Hardy space $H^p_{Mi,w}(\Om)$ is closely related to the Hardy space $H^p_{r,w}(\Om)$ defined by
\[
H^p_{r,w}(\Om)=\{f\in \mathscr{S}': \ \text{there exists $F\in H^p_w(\Rn)$ so that $F|_{\Om}=f$}\}
\]
with the norm
\[
\|f\|_{H^p_{r,w}(\Om)}=\inf\{\|F\|_{H^p_{w}(\Rn)}: F\in H^p_w(\Rn), F|_{\Om}=f\}.
\]
Arguing similarly to \cite{CKS2} we can prove that  if $\Om$ is a bounded Lipschitz domain or  a convex domain above  a Lipschitz graph, then $H^p_{Mi,w}(\Om)\equiv H^p_{w, r}(\Om)$ for $w\in A_\vc(\Rn)$ and $\f{nq_w}{n+1}<p\le 1$. 


Let $p\in (0,1]$, $w\in A_\vc(\Rn)$ and $q\in (q_w,\vc]$. A bounded, measurable function $a$ is called a local $(p,q,w)$-atom if 
\begin{enumerate}[\upshape (i)]
	\item $a$ is supported in a ball $B\subset \mathbb{R}^n$;
	\item $\|a\|_{L^q_w}\le w(B)^{1/q-1/p}$;
	\item $\displaystyle \int_B x^\alpha a(x)dx=0$ 
	for all multi--indices $\alpha$ with $|\alpha|\le \lfloor n(q_w/p-1)\rfloor$, if $r_B<1$.
\end{enumerate}
Similarly to the weighted Hardy space $H^p_w(\Rn)$ we can define the weighted local Hardy spaces via local atomic decompositions for  $p\in (0,1]$ and $w\in A_\vc(\Rn)$, and we denote these local weighted Hardy spaces by $h^p_w(\Rn)$. See for example \cite{H-QBui}.

\begin{defn}
	Let $\Om$ be an open set in $\Rn$. Let $p\in (0,1]$ and $w\in A_\vc(\Rn)$. The weighted  Hardy space $H^p_{z,w}(\Om)$ can be defined as follows:
	\[
	H^p_{z,w}(\Om):=\left\{\begin{array}{cc}
	\{f \in H^p_{w}(\Rn): f\equiv 0 \quad \text{on $\Om^c$}\}, & \text{if $\Om$ is unbounded}\\
	&\\
	\{f \in h^p_{w}(\Rn): f\equiv 0 \quad \text{on $\Om^c$}\}, & \text{if $\Om$ is bounded},
	\end{array}
	\right.
	\]
	with the norm defined by
	\[
	\|f\|_{H^p_{z,w}(\Om)}:=\left\{\begin{array}{cc}
	\|f\|_{H^p_{w}(\Om)}, & \text{if $\Om$ is unbounded}\\
	&\\
	\|f\|_{h^p_{w}(\Om)}, & \text{if $\Om$ is bounded},
	\end{array}
	\right.
	\]
\end{defn}

We remark that if $\Om$ is bounded domain and $w\equiv 1$, the Hardy spaces $H^{p}_{z,w}(\Om)$ coincides with the local  Hardy spaces of extension $h^{p}_{z}(\Om)$ defined in \cite{CKS2}. We also note  that in the case of bounded domain, the local Hardy spaces $h^p_z(\Om)$ defined in \cite{CKS2} and the Hardy spaces $H^p_{CW}(\Om)$ defined by Coifmann and Weiss in \cite{CW} are the equivalent. For this reason we use the same notation $H^p_{z,w}(\Om)$  for both cases of bounded and unbounded domains.
\begin{defn}
	Let $p\in (0,1]$, $w\in A_\vc(\Rn)$ and $q\in (q_w,\vc]$. A bounded, measurable function $a: \Om\to \mathbb{R}$ is called an $(p,q,w)_\Om$-atom if 
	\begin{enumerate}[\upshape (i)]
		\item $a$ is supported in a ball $B\subset \mathbb{R}^n$ and $a\equiv0$ on $\Om^c$;
		\item $\|a\|_{L^q_w(\Om)}\le w(B)^{1/q-1/p}$;
		\item $\displaystyle \int_B x^\alpha a(x)dx=0$ 
		for all multi--indices $\alpha$ with $|\alpha|\le \lfloor n(q_w/p-1)\rfloor$.
	\end{enumerate}
	In the case where $\Om$ is bounded, a function $a$ can be viewed as an atom if 
	\[
	\|a\|_{L^q_w(\Om)}\le w(\Om)^{1/q-1/p}.
	\] 
\end{defn}

The Hardy space $H^{p,q}_{at,w}(\Om)$ is defined as the set of all distributions $f\in \mathscr{S}'$ such that
\[
f=\sum_{j=1}^{\infty}\lambda_j a_j
\]
where $a_j$ are $(p,q,w)$-atoms and $\lambda_j$ are scalars with $\sum_{j}|\lambda_j|^p<\vc$.
We also set
\[
\|f\|^p_{H^{p,q}_{at,w}(\Om)}=\inf\Big\{\sum_{j=1}^{\infty} |\lambda_j|^p: f=\sum_{j=1}^{\infty} \lambda_j a_j \Big\}
\]
where the infimum is taken over all such decompositions.

It is easy to see that for $p\in (0,1]$, $w\in A_\vc(\Rn)$ and $q\in (q_w,\vc]$ we have
\[
H^{p,q}_{at,w}(\Om)\equiv H^{p}_{z,w}(\Om)
\]
for either  $\Om$ is bounded or $\Om$ is unbounded.

In what follows, denote by $\Delta_D$ and $\Delta_N$ the Dirichlet Laplacian and the Neumann Laplacian, respectively. For $w\in A_\vc(\Rn)$ and $0<p\le 1$, we denote by $H^p_{\Delta_D, w}(\Om)$ and $H^p_{\Delta_N, w}(\Om)$ the weighted Hardy spaces associated to $\Delta_D$ and $\Delta_N$, respectively. See the definitions of these function spaces in Section 2. Our first main result concerning the weighted estimates for $\nabla^2 \mathbb{G}_D$ and $\nabla^2 \mathbb{G}_N$ on bounded domains.
\begin{thm}
	\label{mainthm1-bounded domain} Assume that  $\Om\subset \Rn$ is a bounded, simply connected, semiconvex domain (see \cite[Definition 4.6]{DHMMY}) for the Dirichlet problem and $\Om$ is a bounded convex set for the Neumann problem.  Then we have
\begin{enumerate}[{\rm (i)}]
	\item The operators $\nabla^2 \mathbb{G}_D$ and $\nabla^2 \mathbb{G}_N$ extend as bounded operators  on $L^p_w(\Om)$ for all $1<p<2$ and $w\in A_p(\Rn)\cap RH_{(2/p)'}(\Rn)$, and as bounded operators from $L^1_w(\Om)$ into $L^{1,\vc}_w(\Om)$ for $w\in A_1(\Rn)\cap RH_{2}(\Rn)$.\\
	
	\item The operator $\nabla^2 \mathbb{G}_D$ extends as a bounded operator from $H^p_{\Delta_D,w}(\Om)$ into $H^p_{Mi,w}(\Om)$ for all $0<p\le 1$ and $w\in \bigcup_{1<r<2}A_r(\Rn)\cap RH_{(2/r)'}(\Rn)$. As a consequence,  the operator $\nabla^2 \mathbb{G}_D$ extends as a bounded operator on $H^p_{r,w}(\Om)$ for all $\f{n}{n+1}<p\le 1$ and $w\in \bigcup_{1<r<r_0}A_r(\Rn)\cap RH_{(2/r)'}(\Rn)$ where $r_0=\f{np}{n+1}$.\\
	
	\item The operator $\nabla^2 \mathbb{G}_N$ extends as a bounded operator from $H^p_{\Delta_N,w}(\Om)$ into $H^p_{Mi,w}(\Om)$ for all $0<p\le 1$ and $w\in \bigcup_{1<r<2}A_r(\Rn)\cap RH_{(2/r)'}(\Rn)$. As a consequence,  the operator $\nabla^2 \mathbb{G}_D$ extends as a bounded operator from $H^p_{z,w}(\Om)$ into $H^p_{r,w}(\Om)$ for all $\f{n}{n+1}<p\le 1$ and $w\in \bigcup_{1<r<r_0}A_r(\Rn)\cap RH_{(2/r)'}(\Rn)$ where $r_0=\f{np}{n+1}$.
\end{enumerate}	
\end{thm} 
\begin{rem} (a) The main parts of Theorem \ref{mainthm1-bounded domain} are (ii) and (iii).
	
	(b) The $L^p$--weighted estimates in part (i) was obtained in \cite{BD}  for $w\in A_1(\Om)\cap RH_{2}(\Om)$. Here we prove the results (i) in terms of $w\in A_1(\Rn)\cap RH_{2}(\Rn)$. While the class $A_1(\Rn)$ is smaller than $A_1(\Om)$, it is not clear about the two classes $RH_{2}(\Rn)$ and  $RH_{2}(\Om)$. We note that the our proof in this paper can be used to reproduce
	the result for $w\in A_1(\Om)\cap RH_{2}(\Om)$ as well.

 (c) The boundedness of $\nabla^2 \mathbb{G}_D$ and $\nabla^2 \mathbb{G}_N$ from $H^p_{\Delta_D,w}(\Om)$ into $H^p_{Mi,w}(\Om)$ and from $H^p_{\Delta_N,w}(\Om)$ into $H^p_{Mi,w}(\Om)$ for $0<p\le 1$ are new even for $w\equiv 1$. This answers the open question in \cite{HLMMY} for the case $0<p\le \f{n}{n+1}$ under weak smoothness condition on the boundary of the domain. 
 
 (d) The weighted estimates for $\f{n}{n+1}<p\le 1$  was obtained in \cite[Theorems 1.8--1.9]{CCYY} for the local weighted Orlicz-Hardy spaces. However, our estimates are sharper than those in \cite[Theorems 1.8--1.9]{CCYY} since the class of weights $w\in \bigcup_{1<r<r_0}A_r(\Rn)\cap RH_{(2/r)'}(\Rn)$ where $r_0=\f{np}{n+1}$ is strictly larger than those in \cite[Theorems 1.8--1.9]{CCYY} which is $w\in A_q(\Rn)\cap RH_r(\Rn)$ where  
 \[
 q<\f{np}{n+1}, \ r>\f{2}{2-q}\  {\rm and} \  \f{2q}{p}<\f{n+1}{n}+\f{r-1}{pr}.
 \]
 \end{rem}
Our next main result is the following.
\begin{thm}
	\label{mainthm2-above convex Lipschitz domains} Let $\Om\subset \Rn$ be a convex domain above a Lipschitz graph.  Then we have
	\begin{enumerate}[{\rm (i)}]
		\item The operators $\nabla^2 \mathbb{G}_D$ and $\nabla^2 \mathbb{G}_N$ extend as bounded operators  on $L^p_w(\Om)$ for all $1<p<2$ and $w\in A_p(\Rn)(\Rn)\cap RH_{(2/p)'}(\Rn)$, and as bounded operators from $L^1_w(\Om)$ into $L^{1,\vc}_w(\Om)$ for $w\in A_1(\Rn)\cap RH_{2}(\Rn)$.\\
		
		\item The operator $\nabla^2 \mathbb{G}_D$ extends as a bounded operator from $H^p_{\Delta_D,w}(\Om)$ into $H^p_{Mi,w}(\Om)$ for all $0<p\le 1$ and $w\in \bigcup_{1<r<2}A_r(\Rn)\cap RH_{(2/r)'}(\Rn)$. As a consequence,  the operator $\nabla^2 \mathbb{G}_D$ extends as a bounded operator on $H^p_{r,w}(\Om)$ for all $\f{n}{n+1}<p\le 1$ and $w\in \bigcup_{1<r<r_0}A_r(\Rn)\cap RH_{(2/r)'}(\Rn)$  where $r_0=\f{np}{n+1}$.\\
		
		\item The operator $\nabla^2 \mathbb{G}_N$ extends as a bounded operator from $H^p_{\Delta_N,w}(\Om)$ into $H^p_{Mi,w}(\Om)$ for all $0<p\le 1$ and $w\in \bigcup_{1<r<2}A_r(\Rn)\cap RH_{(2/r)'}(\Rn)$. As a consequence,  the operator $\nabla^2 \mathbb{G}_D$ extends as a bounded operator from $H^p_{z,w}(\Om)$ into $H^p_{r,w}(\Om)$ for all $\f{n}{n+1}<p\le 1$ and $w\in \bigcup_{1<r<r_0}A_r(\Rn)\cap RH_{(2/r)'}(\Rn)$  where $r_0=\f{np}{n+1}$.
	\end{enumerate}	
\end{thm}  

\begin{rem} The $L^p$--boundedness  from a suitable Hardy space into $L^1$ for $\nabla^2 \mathbb{G}_D$ and $\nabla^2 \mathbb{G}_N$ for $1<p<2$  when $\Om\subset \Rn$ is a convex domain above a Lipschitz graph was obtained in \cite{A2, AJ}. The results in Theorem \ref{mainthm2-above convex Lipschitz domains} are new; moreover the results in (ii) and (iii) are new even for unweighted cases.
	\end{rem}
	In the case of half spaces, in addition to Theorem \ref{mainthm2-above convex Lipschitz domains}, we can have further estimates as follows.
\begin{thm}
	\label{mainthm3-upper half-space} Let $\Om= \Rn_+$ be the upper half-space. Then we have
	\begin{enumerate}[{\rm (i)}]
		\item The operators $\nabla^2 \mathbb{G}_D$ and $\nabla^2 \mathbb{G}_N$ extend as bounded operators  on $L^p_w(\Om)$ for all $1<p<\vc$ and $w\in A_p(\Rn)$, and as bounded operators from $L^1_w(\Om)$ into $L^{1,\vc}_w(\Om)$ for $w\in A_1$.\\
		
		\item The operator $\nabla^2 \mathbb{G}_D$ extends as a bounded operator from $H^p_{\Delta_D,w}(\Om)$ into $H^p_{Mi,w}(\Om)$ for all $0<p\le 1$ and $w\in A_\vc(\Rn)$. As a consequence,  the operator $\nabla^2 \mathbb{G}_D$ extends as a bounded operator on $H^p_{r,w}(\Om)$ for all $\f{n}{n+1}<p\le 1$ and $w\in \bigcup_{1<r<r_0}A_r(\Rn)\cap RH_{(2/r)'}(\Rn)$  where $r_0=\f{np}{n+1}$.\\
		
		\item The operator $\nabla^2 \mathbb{G}_D$ extends as a bounded operator from $H^p_{z,w}(\Om)$ into $H^p_{r,w}(\Om)$ for all $0<p\le 1$ and $w\in A_\vc(\Rn)$.\\
		
		\item The operator $\nabla^2 \mathbb{G}_N$ extends as a bounded operator from $H^p_{\Delta_N,w}(\Om)$ into $H^p_{Mi,w}(\Om)$ for all $0<p\le 1$ and $w\in A_\vc(\Rn)$. As a consequence,  the operator $\nabla^2 \mathbb{G}_D$ extends as a bounded operator from $H^p_{z,w}(\Om)$ into $H^p_{r,w}(\Om)$ for all $0<p\le 1$ and $w\in A_\vc(\Rn)$.
	\end{enumerate}	
\end{thm}  
\begin{rem}
In \cite{CKS2}, it was proved that $\nabla^2 \mathbb{G}_D$ is bounded on $H^p_{r}(\Om)$ for all $\f{n}{n+1}<p\le 1$, and $\nabla^2 \mathbb{G}_D$ and $\nabla^2 \mathbb{G}_N$ are bounded from $H^p_z(\Om)$ into $H^p_{r}(\Om)$ for all $0<p\le 1$. Hence, our results can be viewed as an extension of those in \cite{CKS2} to the weighted estimates, meanwhile the boundedness from $H^p_{\Delta_D,w}(\Om)$ into $H^p_{Mi,w}(\Om)$ for  $\nabla^2 \mathbb{G}_D$ are new even for unweighted case.
\end{rem}
\medskip

Our approach relies on the theory of Hardy spaces associated to operators which was initially developed in \cite{ADM} and has been studied intensively by many authors. See for example \cite{DY, HM, HLMMY} and the references therein. We first prove that the Hardy spaces defined via atomic decompositions and maximal functions are equivalent. See Theorem \ref{maximal function result}. This plays a crucial role in the proof of our main results and is interesting in its own right. We note that \cite{SY, SY2}
showed unweighted estimates when the underlying space has infinite measure.

We also remark that the atomic decompositions for the Hardy spaces were obtained in \cite{DHMMY, CCYY} by using the existing atomic decomposition results for the tent spaces. However, it seems that the approach in \cite{DHMMY, CCYY} is not applicable to our setting when the domain is bounded since the atomic decomposition results for the tent spaces might not be true for the bounded domains. To overcome this trouble, we adapt some ideas in \cite{BDK2} which makes use of kernel estimates for functional calculus, estimates for maximal functions and the Whitney covering lemma. Note that our approach can be easily applied to study the problems in the Musielak--Hardy spaces and this might be done elsewhere.

The organization of our paper is as follows. In Section 2, we prove the equivalence between the atomic Hardy spaces and maximal Hardy spaces in the general setting of spaces of homogeneous type. This result is interesting in its own right. The proofs of the main results will be addressed in Section 3.

\bigskip

\textbf{Notation.} As usual we use $C$ and $c$ to denote positive constants that are independent of the main parameters involved but may differ from line to line. The notation $A\lesi B$  means $A\leq CB$, and $A\sim B$ means that both $A\lesi B$ and $B\lesi A$ hold. We use $\fint_E fd\mu=\f{1}{\mu(E)}\int_E fd\mu$ to denote the average of $f$ over $E$.
We write $B(x,r)$ to denote the ball centred at $x$ with radius $r$. By a `ball $B$' we mean the ball $B(x_B, r_B)$ with some fixed centre $x_B$ and radius $r_B$. The annuli around a  given ball $B$ will be denoted by $S_j(B)=2^{j+1}B\backslash 2^jB$ for $j\ge 1$ and $S_0(B)=2B$ for $j=0$. 
       
\section{Weighted Hardy spaces associated to operators}

In this section, we study weighted Hardy spaces on a general space of homogeneous type $X$ which is of interest in its own right and has a 
doubling domain $\Omega \subset  \mathbb R^n$ as a special case.

Let $(X,d, \mu)$ be a metric space endowed with a nonnegative Borel measure $\mu$ satisfying the doubling condition: there exists a constant $C_1>0$ such that
\begin{equation}\label{doublingcondition}
\mu(B(x,2r))\leq C_1\mu(B(x,r))
\end{equation}
for all $x\in X$, $r>0$ and all balls $B(x,r):=\{y\in X: d(x,y)<r\}$. For the moment  $\mu(X)$ may be finite or infinite. 

It is not difficult to see that the condition \eqref{doublingcondition} implies that there exists a ``dimensional" constant $n\geq 0$ so that
\begin{equation}\label{doub2}
\mu(B(x,\lambda r))\leq C_2\lambda^n \mu(B(x,r))
\end{equation}
for all $x\in X, r>0$ and $\lambda\geq 1$, and
\begin{equation}\label{doub2s}
\mu(B(x, r))\leq C_3\mu(B(y,r))\Big(1+\f{d(x,y)}{r}\Big)^n
\end{equation}
for all $x,y\in X, r>0$.

A weight $w$ is a non-negative measurable and locally integrable function on $X$.
We say that $w$ belongs to the Muckenhoupt class $A_p(X)$ for $1 < p < \infty$, if there exists a
constant $C$ such that for every ball $B \subset X$,
$$
\Big(\fint_B w(x)\dx\Big)\Big(\fint_B w^{-1/(p-1)}(x)\dx\Big)^{p-1}\leq C.
$$
For $p = 1$, we say that $w \in A_1(X)$ if there is a constant $C$ such
that for every ball $B \subset X$,
$$
\fint_B w(y)\dy \leq Cw(x) \ \text{for a.e. $x\in B$}.
$$
We define $A_\vc(X)=\cup_{p\geq 1}A_p(X)$.\\

The reverse H\"older classes of weights $RH_q$ are defined in the following way: $w
\in RH_q, 1 < q < \infty$, if there is a constant $C$ such that for
any ball $B \subset X$,
$$
\Big(\fint_B w^q(y) \dy\Big)^{1/q} \leq C \fint_B w(x)\dx.
$$
The endpoint $q = \infty$ is given by the condition: $w \in
RH_\infty$ whenever, there is a constant $C$ such that for any ball
$B \subset X$,
$$
w(x)\leq C \fint_B w(y)\dy  \ \text{for a.e. $x\in B$}.
$$
Let $w \in A_\vc(\Rn)$, for $1\leq p <\infty$, the weighted spaces $L^p_w(X)$
can be defined by
$$\Big\{f :\int_{X} |f(x)|^p w(x)\dx < \infty\Big\}$$
with the norm
$$\|f\|_{L^p_w(X)}=\Big(\int_{X} |f(x)|^p w(x)\dx\Big)^{1/p}.$$

We sum up some of the standard properties of classes of weights  in the following lemma. For the proofs, see for example \cite{ST}.
\begin{lem}\label{weightedlemma1}
	The following properties hold:
	\begin{enumerate}[(i)]
		\item $A_1(X)\subset A_p(X)\subset A_q(X)$ for $1< p\leq q< \infty$.
		\item $RH_\infty (X) \su RH_q(X) \su RH_p(X)$ for $1< p\leq q< \infty$.
		\item If $w \in A_p(X), 1 < p < \vc$, then there exists $1<r < p < \vc$ such that $w \in A_r$.
		\item If $w \in RH_q(X), 1 < q < \vc$, then there exists $q < p < \vc$ such that $w \in RH_p(X)$.
		\item $ A_\vc(X) =\cup_{1\leq p<\vc}A_p(X)\subset  \cup_{1< q\leq \vc}RH_q(X)$
	\end{enumerate}
\end{lem}
For $w\in A_\vc(\Rn)$ we define 
\begin{equation}
\label{qw-defn}
q_w=\sup\{p\in (1,\vc): w\in A_p(X)\}.
\end{equation}

In this paper, we will also assume  the existence of an operator $L$ that satisfies the following two conditions:
\begin{enumerate}
	\item[(A1)] $L$ is  a nonnegative self-adjoint operator on $L^2(X)$;
	\item[(A2)] $L$ generates a semigroup $\{e^{-tL}\}_{t>0}$ whose kernel $p_t(x,y)$ admits a Gaussian upper bound. That is, there exist two positive constants $C$ and  $c$ so that for all $x,y\in X$ and $t>0$,
	\begin{equation}
	\tag{GE}\label{GE}
	\displaystyle |p_t(x,y)|\leq \f{C}{\mu(B(x,\sqrt{t}))}\exp\Big(-\f{d(x,y)^2}{ct}\Big).
	\end{equation}
\end{enumerate}
Then for $0<p\le 1$ one can define three types of Hardy spaces related to $L$. The first type is through linear combinations of atoms that appropriately encode the cancellation inherent in $L$. The second and third types  are $H^p_{L,\max}$ and $H^p_{L,\rad}$, which are defined via the non-tangential maximal function and the radial maximal function respectively. For the reader's  convenience, we recall these Hardy spaces below.

\begin{defn}[Atoms for $L$]\label{def: L-atom}
	Let $p\in (0,1]$, $q\in (1,\vc]$, $M\in \mathbb{N}$ and $w\in A_\vc(X)$. A function $a$ supported in a ball $B$ is called an  $(L,p,q,w,M)$-atom if there exists a
	function $b\in {\mathcal D}(L^M)$ such that
	\begin{enumerate}[{\rm (i)}]
		\item  $a=L^M b$;
		\item $\supp L ^{k}b\subset B, \ k=0, 1, \dots, M$;
		\item $\|L^{k}b\|_{L^q_w(X)}\leq
		r_B^{2(M-k)}w(B)^{\f{1}{q}-\f{1}{p}},\ k=0,1,\dots,M$.
	\end{enumerate}
	In the particular case where $\mu(X)<\vc$, the constant function $[w(X)]^{-1/p}$ is also considered as an atom.
\end{defn}

In contrast to the concept of atoms in \cite{CCYY} in which the atoms are defined via $L^q$-norms, in Definition \ref{def: L-atom} our atoms  are defined via the weighted $L^q$-norms. This plays an essential role in proving the weighted Hardy estimates for the Green operators.  

Given $p\in (0,1]$, $q\in (1,\vc]$, $M\in \mathbb{N}$ and $w\in A_\vc(X)$, we  say that $f=\sum
_{j=1}^{\infty}\lambda_ja_j$ is an atomic $(L,p,q,w,M)$-representation if
$\{\lambda_j\}_{j=0}^\infty\in l^p$, each $a_j$ is a $(L,p,q,w,M)$-atom,
and the sum converges in $L^2(X)$. The space $H^{p,q}_{L,w,at,M}(X)$ is then defined as the completion of
\[
\left\{f\in L^2(X):f \ \text{has an atomic
	$(L,p,q,w,M)$-representation}\right\},
\]
with the norm given by
$$
\|f\|_{H^{p,q}_{L,w,at,M}(X)}=\inf\left\{\left(\sum_{j=1}^{\infty} |\lambda_j|^p\right)^{1/p}:
f=\sum_{j=1}^{\infty} \lambda_ja_j \ \text{is an atomic $(L,p,q,w,M)$-representation}\right\}.
$$

\begin{defn}[Maximal Hardy spaces for $L$]\label{defn-maximal Hardy spaces}
	
	For $f\in L^2(X)$, we define the \textit{non-tangential}  maximal function associated to $L$ of $f$ by
	\[
	f^*_{L}(x)=\sup_{0<t<d^2_X}\sup_{d(x,y)<t}|e^{-tL}f(y)|
	\]
	and the \textit{radial} maximal function by
	\[
	f^+_{L}(x)=\sup_{0<t<d^2_X}|e^{-tL}f(x)|
	\]
    where $d_X = {\rm diam\,}X$.

	Given $p\in (0,1]$ and $w\in A_\vc(X)$,  the Hardy space $H^{p}_{L, w,{\rm max}}(X)$ is defined as the completion of
	$$
	\left\{f
	\in L^2(X): f^*_{L} \in L^p_w(X)\right\},
	$$
	with the norm given by
	$$
	\|f\|_{H^{p}_{L, w,{\rm max}}(X)}=\|f^*_{L}\|_{L^p_w(X)}.
	$$
	
	Similarly, the Hardy space $H^{p}_{L, w,{\rm rad}}(X)$ is defined as the completion of
	$$
	\left\{f
	\in L^2(X): f^+_{L} \in L^p_w(X)\right\},
	$$
	with the norm given by
	$$
	\|f\|_{H^{p}_{L,w,{\rm rad}}(X)}=\|f^+_{L}\|_{L_w^p(X)}.
	$$
\end{defn}

 We will show that the three types of Hardy space are equivalent as in the following result.
\begin{thm}\label{maximal function result}
	Let $X$ be a space of homogeneous type with finite or infinite measure and let $L$ be an operator satisfying (A1) and (A2). Let $p\in (0,1]$, $w\in A_\vc(X)$, $q\in (q_w,\vc]$, and $M>\f{n}{2}\big(\f{q_w}{p}-1\big)$. Then the Hardy spaces $H^{p,q}_{L,w,at,M}(X)$, $H^{p}_{L, w,{\rm max}}(X)$ and $H^{p}_{L,w, {\rm rad}}(X)$ coincide with equivalent norms. 
\end{thm}
We note that the unweighted case of the theorem was proved in \cite{BDK2}. The proof of Theorem
\ref{maximal function result} will be given at the end of this section.

\subsection{Some maximal function estimates}\label{sect-prelim}

Let  $L$  satisfy (A1) and (A2). Denote by $E_L(\lambda)$ the spectral decomposition of $L$. Then by spectral theory, for any bounded Borel funtion $F:[0,\vc)\to \mathbb{C}$ we can define
$$
F(L)=\int_0^\vc F(\lambda)dE_L(\lambda)
$$
as a bounded operator on $L^2(X)$. It is well-known that the kernel $K_{\cos(t\sqrt{L})}$ of $\cos(t\sqrt{L})$ satisfies 
\begin{equation}\label{finitepropagation}
{\rm supp}\,K_{\cos(t\sqrt{L})}\subset \{(x,y)\in X\times X:
d(x,y)\leq t\}.
\end{equation}
See for example \cite{CS}.
We have the following useful lemmas.
\begin{lem}[\cite{HLMMY}]\label{lem:finite propagation}
	Let $\varphi\in C^\vc_0(\mathbb{R})$ be an even function with {\rm supp}\,$\varphi\subset (-1, 1)$ and $\int \varphi =2\pi$. Denote by $\Phi$ the Fourier transform of $\varphi$.  Then for every $k\in \mathbb{N}$, the kernel $K_{(t\sqrt{L})^k\Phi(t\sqrt{L})}(\cdot,\cdot)$ of $(t\sqrt{L})^k\Phi^{(\ell)}(t\sqrt{L})$ satisfies 
	\begin{equation}\label{eq1-lemPsiL}
	\displaystyle
	{\rm supp}\,K_{(t\sqrt{L})^k\Phi(t\sqrt{L})}\subset \{(x,y)\in X\times X:
	d(x,y)\leq t\},
	\end{equation}
	and
	\begin{equation}\label{eq2-lemPsiL}
	|K_{(t\sqrt{L})^k\Phi(t\sqrt{L})}(x,y)|\leq \f{C}{\mu(B(x,t))}.
	\end{equation}
\end{lem}
\begin{lem}
	\label{lem1}
	\begin{enumerate}[{\rm (a)}]
		\item Let $\varphi\in \mathscr{S}(\mathbb{R})$ be an even function. Then for any $N>0$ there exists $C$ such that 
		\begin{equation}
		\label{eq1-lema1}
		|K_{\varphi(t\sqrt{L})}(x,y)|\leq \f{C}{\mu(B(x,t))+\mu(B(y,t))}\Big(1+\f{d(x,y)}{t}\Big)^{-n-N},
		\end{equation}
		for all $t>0$ and $x,y\in X$.
		\item Let $\varphi_1, \varphi_2\in \mathscr{S}(\mathbb{R})$ be even functions. Then for any $N>0$ there exists $C$ such that
		\begin{equation}
		\label{eq2-lema1}
		|K_{\varphi_1(t\sqrt{L})\varphi_2(s\sqrt{L})}(x,y)|\leq C\f{1}{\mu(B(x,t))+\mu(B(y,t))}\Big(1+\f{d(x,y)}{t}\Big)^{-n-N},
		\end{equation}
		for all $t\leq s<2t$ and $x,y\in X$.
		\item Let $\varphi_1, \varphi_2\in \mathscr{S}(\mathbb{R})$ be even functions with $\varphi^{(\nu)}_2(0)=0$ for $\nu=0,1,\ldots,2\ell$ for some $\ell\in\mathbb{Z}^+$. Then for any $N>0$ there exists $C$ such that
		\begin{equation}
		\label{eq3-lema1}
		|K_{\varphi_1(t\sqrt{L})\varphi_2(s\sqrt{L})}(x,y)|\leq C\Big(\f{s}{t}\Big)^{2\ell} \f{1}{\mu(B(x,t))+\mu(B(y,t))}\Big(1+\f{d(x,y)}{t}\Big)^{-n-N},
		\end{equation}
		for all $t\geq s>0$ and $x,y\in X$.
	\end{enumerate}
\end{lem}
\begin{proof}
	\noindent (a) The estimate \eqref{eq1-lema1} was proved in \cite[Lemma 2.3]{CD} in the particular case $X=\mathbb{R}^n$ but the proof is still valid in the spaces of homogeneous type. For the items (b) and (c) we refer to \cite{BDK}.
\end{proof}

We record the following result in \cite{DKP}.
\begin{lem}
	\label{lem-DKP}
	Let $\varphi\in\mathscr{S}(\mathbb{R})$ be even function with $\varphi(0)=1$ and let $N>0$. Then there exist even functions $\phi,\psi\in \mathscr{S}(\mathbb{R})$ with $\phi(0)=1$ and $ \psi^{(\nu)}(0)=0, \nu=0,1,\ldots, N$ so that for every $f\in L^2(X)$ and every $j\in \mathbb{Z}$ we have
	$$
	f=\phi(2^{-j}\sqrt{L})\varphi(2^{-j}\sqrt{L})f+\sum_{k\geq j}\psi(2^{-k}\sqrt{L})[\varphi(2^{-k}\sqrt{L})-\varphi(2^{-k+1}\sqrt{L})]f \ \ \text{in $L^2(X)$}.
	$$
\end{lem}

The following elementary estimate will be used frequently. Its proof is simple and we omit it.
\begin{lem}\label{lem-ele est}
	Let $\epsilon >0$. We have
	$$
	\int_X\f{1}{\mu(B(x,s))\wedge \mu(B(y,s))}\Big(1+\f{d(x,y)}{s}\Big)^{-n-\epsilon}|f(y)|\dy\lesi \mathcal{M}f(x).
	$$
	for all $x\in X$, $s>0$ where $\mathcal{M}f(x)$ is the Hardy-Littlewood maximal function of $f$.
\end{lem}

Let $F$ be a measurable function on $X\times (0,\vc)$. For $\alpha>0$ we set 
$$
F^*_\alpha(x)=\sup_{0<t<d_X}\sup_{d(x,y)<\alpha t}|F(y,t)|.
$$
In the particular case $\alpha=1$, we write $F^*$ instead of $F^*_\alpha$.\\

We have the following result:
\begin{lem}
	\label{lem2} For any $p>0$, $w\in A_q$ and $0<\alpha_2\leq \alpha_1$, there exists $C$ depending on $n$ and $p$ so that
	\[
	\|F^*_{\alpha_1}\|_{L_w^p(X)}\leq C\Big(1+\f{2\alpha_1}{\alpha_2}\Big)^{nq/p}\|F^*_{\alpha_2}\|_{L_w^p(X)}.
	\]
\end{lem}
\begin{proof}
	The proof of this lemma is  similar to that of \cite[Theorem 2.3]{CT}, hence we just sketch the main ideas.
	
	Set 
	$$
	E_1=\{x\in X: F^*_{\alpha_1}(x)>\lambda\}, \ \ \ \text{and} \ \ \ E_2=\{x\in X: F^*_{\alpha_2}(x)>\lambda\}.
	$$
	Then if $x_0\in E_1$, arguing similarly to the proof of \cite[Theorem 2.3]{CT} we can find $0<t_0<d_X$ so that
	\[
	\f{\mu(B(x_0,(\alpha_1+\alpha_2)t_0)\cap E_2)}{V(x_0,(\alpha_1+\alpha_2)t_0)}\ge  C\Big(1+\f{2\alpha_1}{\alpha_2}\Big)^{-n}.
	\]
	This implies
	\[
		\f{w(B(x_0,(\alpha_1+\alpha_2)t_0)\cap E_2)}{w(B(x_0,(\alpha_1+\alpha_2)t_0))}\ge  C_0\Big(1+\f{2\alpha_1}{\alpha_2}\Big)^{-nq}.
	\]
	As a consequence, we have
	\[
	E_1\subset \left\{ x\in X: \mathcal{M}_w(\chi_{E_2})(x)>C_0\Big(1+\f{2\alpha_1}{\alpha_2}\Big)^{-nq} \right\}.
	\] 
	This, along with the weak type $(1,1)$ of the maximal function $\mathcal{M}_w$, yields
	\[
	w(E_1)\le C_1\Big(1+\f{2\alpha_1}{\alpha_2}\Big)^{nq}w(E_2),
	\]
	or equivalently,
	\begin{equation*}
	\label{ E1 E2}
	w\{x\in X: F^*_{\alpha_1}(x)>\lambda\}\le C_1\Big(1+\f{2\alpha_1}{\alpha_2}\Big)^{nq} w\{x\in X: F^*_{\alpha_2}(x)>\lambda\}.
	\end{equation*}
	Therefore, 
	\[
	\begin{aligned}
	\|F^*_{\alpha_1}\|^p_{L_w^p(X)}&=p\int_0^p\lambda^{p-1}w\{x\in X: F^*_{\alpha_1}(x)>\lambda\}\,d\lambda\\
	&\le pC_1\Big(1+\f{2\alpha_1}{\alpha_2}\Big)^{nq} \int_0^p\lambda^{p-1}w\{x\in X: F^*_{\alpha_2}(x)>\lambda\}\,d\lambda\\
	&\le pC_1\Big(1+\f{2\alpha_1}{\alpha_2}\Big)^{nq}\|F^*_{\alpha_2}\|^p_{L_w^p(X)}.
	\end{aligned}
	\]
	This completes our proof.
\end{proof}
From the lemmas above we obtain the following result.
\begin{lem}
	\label{lem3} For any $p\in (0,1]$, $w\in A_q(X)$ and $\lambda>nq/p$, there exists $C$ depending on $n$ and $p$ so that
	\[
	\Big\|\sup_{0<t<d_X}\sup_{y}F(y,t)\Big(1+\f{d(x,y)}{t}\Big)^{-\lambda}\Big\|_{L_w^p(X)}\leq C\|F^*\|_{L_w^p(X)}.
	\]
\end{lem}

For any even function  $\varphi \in \mathscr{S}(\mathbb{R})$, $\alpha>0$ and $f\in L^2(X)$ we define
$$
\varphi^*_{L,\alpha}(f)(x)=\sup_{0<t<d_X}\sup_{d(x,y)<\alpha t}|\varphi(t\sqrt{L})f(y)|,
$$ 
and 
$$
\varphi^+_{L,\alpha}(f)(x)=\sup_{0<t<d_X}|\varphi(t\sqrt{L})f(x)|.
$$
As usual, we drop the index $\alpha$ when $\alpha=1$.

We now are in position to prove the following estimate.

\begin{prop}
	\label{prop1}
	Let $p\in (0,1]$ and $w\in A_\vc(X)$. Let $\varphi_1, \varphi_2\in \mathscr{\mathbb{R}}$ be even functions with $\varphi_1(0)=1$ and $\varphi_2(0)=0$ and $\alpha_1, \alpha_2>0$. Then for every $f\in L^2(X)$ we have
	\begin{equation}
	\label{eq1-prop1}
	\|(\varphi_2)^*_{L,\alpha_2}f\|_{L_w^p(X)}\lesi \|(\varphi_1)^*_{L,\alpha_1}f\|_{L_w^p(X)}.
	\end{equation}
	As a consequence, for every even function $\varphi$ with $\varphi(0)=1$ and $\alpha>0$ we have
	\begin{equation}
	\label{eq2-prop1}
	\|\varphi^*_{L,\alpha}f\|_{L_w^p(X)}\sim \|f^*_{L}\|_{L_w^p(X)}.
	\end{equation}
\end{prop}
\begin{proof} From Lemma \ref{lem2} it suffices to prove the proposition with $\alpha_1=\alpha_2=1$. 
	
	Fix  $N>n$ and $\lambda>nq_w/p$ and $M>\lambda/2$. Fix $t\in (0,d_X)$ and let $j_0\in \mathbb{Z}^+$ so that $2^{-j_0+1}\leq t<2^{-j_0+2}$. According to Lemma \ref{lem-DKP} there exist even functions $\phi,\psi\in \mathscr{\mathbb{R}}$ with $\phi(0)=1$ and $\psi^{(\nu)}(0)=0$ for $\nu=0,1,\ldots,2M$  so that
	\[
	f=\phi(2^{-j_0}\sqrt{L})\varphi_1(2^{-j_0}\sqrt{L})f+\sum_{k\geq j_0}\psi(2^{-k}\sqrt{L})[\varphi_1(2^{-k}\sqrt{L})-\varphi_1(2^{-k+1}\sqrt{L})]
	\]
	which implies
	\begin{equation}
	\label{eq1-prop1-phi12}
	\begin{aligned}
	\varphi_2(t\sqrt{L})f(y) &=\varphi_2(t\sqrt{L})\phi(2^{-j_0}\sqrt{L})\varphi_1(2^{-j_0}\sqrt{L})f(y)\\
	&\qquad +\sum_{k\geq j_0}\varphi_2(t\sqrt{L})\psi(2^{-k}\sqrt{L})[\varphi_1(2^{-k}\sqrt{L})-\varphi_1(2^{-k+1}\sqrt{L})]f(y)\\
	&=:I_1(y,t)+I_2(y,t).
	\end{aligned}
		\end{equation}
	Since $2^{-j_0}\sim t$, by Lemma \ref{lem1} we have
	\[
	\begin{aligned}
	|I_1(y,t)|\lesi  \int_X  \f{1}{V(z,t)}\Big(1+\f{d(y,z)}{t}\Big)^{-\lambda-N} |\varphi_1(t\sqrt{L})f(z)|d\mu(z).
	\end{aligned}
	\]
	This implies that for a fixed $x\in X$ and $0<t<d_X$ we have
	\[
	\begin{aligned}
	\sup_{d(x,y)<t}|I_1(y,t)|&\lesi \int_X  \f{1}{V(z,t)}\Big(1+\f{d(x,z)}{t}\Big)^{-\lambda-N} |\varphi_1(t\sqrt{L})f(z)|d\mu(z).
	\end{aligned}
	\]
	Applying the estimate in Lemma \ref{lem-ele est} we have
	\begin{equation}
	\label{eq2-prop1-phi12}
	\begin{aligned}
	\sup_{d(x,y)<t}|I_1(y,t)|
	&\lesi  \sup_{z\in X}\Big(1+\f{d(x,z)}{t}\Big)^{-\lambda} |\varphi_1(t\sqrt{L})f(z)|.
	\end{aligned}
	\end{equation}
	
	For the second term $I_2$, using  \eqref{eq3-lema1} and the fact that $t\sim 2^{-j_0}$ we have
	\[
	\begin{aligned}
	|I_2(y,t)|\lesi  \ & \sum_{k\ge j_0} \int_X 2^{-2M(k-j_0)}\f{1}{V(z,2^{-j_0})}\Big(1+\f{d(y,z)}{2^{-j_0}}\Big)^{-\lambda-N} |\varphi_1(2^{-k}\sqrt{L})f(z)|d\mu(z)\\
	& +\sum_{k\ge j_0} \int_X 2^{-2M(k-j_0)}\f{1}{V(z,2^{-j_0})}\Big(1+\f{d(y,z)}{2^{-j_0}}\Big)^{-\lambda-N} |\varphi_1(2^{-k+1}\sqrt{L})f(z)|d\mu(z)\\
	\end{aligned}
	\]
	Hence, for a fixed $x\in X$ and $0<t<d_X$ we have
	\begin{equation*}
	\begin{aligned}
	\sup_{d(x,y)<t}&|I_2(y,t)|\\
	\lesi  \ & \sum_{k\ge j_0} \int_X 2^{-(2M-\lambda)(k-j_0)}\f{1}{V(z,2^{-j_0})}\Big(1+\f{d(x,z)}{2^{-j_0}}\Big)^{-N}\Big(1+\f{d(x,z)}{2^{-k}}\Big)^{-\lambda} |\varphi_1(2^{-k}\sqrt{L})f(z)|d\mu(z)\\
	& +\sum_{k\ge j_0} \int_X 2^{-2M(k-j_0)}\f{1}{V(z,2^{-j_0})}\Big(1+\f{d(x,z)}{2^{-j_0}}\Big)^{-N}\Big(1+\f{d(x,z)}{2^{-k+1}}\Big)^{-\lambda} |\varphi_1(2^{-k+1}\sqrt{L})f(z)|d\mu(z).
	\end{aligned}
	\end{equation*}
    Note that for $k\ge j_0$ we have $2^{-k+1}\le t <d_X$. This, along with Lemma \ref{lem-ele est} and above inequality, implies
	\begin{equation}
	\label{eq3-prop1-phi12}
	\begin{aligned}
	\sup_{d(x,y)<t}|I_2(y,t)|&\lesi  \sup_{{0<s<d_X}}\sup_{z\in X}\Big(1+\f{d(x,z)}{s}\Big)^{-\lambda} |\varphi_1(s\sqrt{L})f(z)|.
	\end{aligned}
	\end{equation}
	Taking this, \eqref{eq2-prop1-phi12} and \eqref{eq1-prop1-phi12} into account we conclude that
	\[
	\sup_{0<t<d_X}\sup_{d(x,y)<t}|\varphi_2(t\sqrt{L})f(y)|\lesi  \sup_{{0<s<d_X}}\sup_{z\in X}\Big(1+\f{d(x,z)}{s}\Big)^{-\lambda} |\varphi_1(s\sqrt{L})f(z)|.
	\]
	Then applying Lemma \ref{lem3}, \eqref{eq1-prop1} follows directly.
	
	To prove \eqref{eq2-prop1}, we apply \eqref{eq1-prop1} for $\varphi_1(\lambda)=\varphi(\lambda)-e^{-\lambda^2}$, $\varphi_2(\lambda)=e^{-\lambda^2}$, $\alpha_1=
	\alpha$ and $\alpha_2=1$ to obtain
	$$
	\Big\|\sup_{t>0}\sup_{d(x,y)<\alpha t}|\varphi(t\sqrt{L})f(y)-e^{-t^2L}f(y)|\Big\|_{L^p_w(X)}\lesi \|f^*_{L}\|_{L^p_w(X)}.
	$$
	This, along with Lemma \ref{lem2}, yields 
	$$
	\|\varphi^*_{L,\alpha}f\|_{L^p_w(X)}\lesi \|f^*_{L}\|_{L^p_w(X)}.
	$$
	Similarly, we obtain
	$$
	\|f^*_{L}\|_{L^p_w(X)}\lesi \|\varphi^*_{L,\alpha}f\|_{L^p_w(X)}.
	$$
	This proves \eqref{eq2-prop1}.
\end{proof}

For each $\lambda>0$ and each even function $\varphi\in\mathscr{S}(\mathbb{R})$ we define
$$
M^*_{L,\varphi,\lambda}f(x)=\sup_{t>0}\sup_{y\in X}\f{|\varphi(t\sqrt{L})f(y)|}{\Big(1+\f{d(x,y)}{t}\Big)^\lambda},
$$
for each $f\in L^2(X)$.

Obviously, we have $\varphi^*_L f(x)\leq M^*_{L,\varphi,\lambda}f(x)$ for all $x\in X, \lambda>0$ and even functions $\varphi\in\mathscr{S}(\mathbb{R})$.

\begin{prop}
	\label{prop2}
	Let $p\in(0,1]$ and $w\in A_\vc(X)$. Let $\varphi\in\mathscr{S}(\mathbb{R})$ be an even function with $\varphi(0)=1$. Then we have, for every $f\in L^2(X)$,
	\begin{equation}\label{eq-petree}
	\Big\|M^*_{L,\varphi,\lambda}f\Big\|_{L_w^p(X)}\lesi \|\varphi^+_{L}f\|_{L_w^p(X)},
	\end{equation}
	provided $\lambda>nq_w/p$.
	
	As a consequence, we have
	\begin{equation*}
	\Big\|\varphi^*_{L}f\Big\|_{L_w^p(X)}\lesi \|\varphi^+_{L}f\|_{L_w^p(X)}.
	\end{equation*}
\end{prop}
\begin{proof}
	Fix  $N>n$ and $\lambda>nq_w/p$ and $M>\lambda/2$. Fix $\theta \in (0,p)$ so that $p/\theta> q_w$ and $\lambda >n/\theta$. We now rewrite \eqref{eq1-prop1-phi12}:
	\begin{equation*}
		\begin{aligned}
	\varphi_2(t\sqrt{L})f(y) &=\varphi_2(t\sqrt{L})\phi(2^{-j_0}\sqrt{L})\varphi_1(2^{-j_0}\sqrt{L})f(y)\\
	&\qquad +\sum_{k\geq j_0}\varphi_2(t\sqrt{L})\psi(2^{-k}\sqrt{L})[\varphi_1(2^{-k}\sqrt{L})-\varphi_1(2^{-k+1}\sqrt{L})]f(y)\\
	&=:I_1(y,t)+I_2(y,t).
	\end{aligned}
	\end{equation*}
	Arguing similarly to \eqref{eq2-prop1-phi12} we have, for a fixed $x\in X$, all $y\in X$ and $t>0$,
	\begin{equation}
	\begin{aligned}
	\Big(1+&\f{d(x,y)}{t}\Big)^{-\lambda}|I_1(y,t)|\\
	&\lesi  \int_X\f{1}{V(z,t)}\Big(1+\f{d(y,z)}{t}\Big)^{-\lambda}\Big(1+\f{d(x,y)}{t}\Big)^{-\lambda} |\varphi(t\sqrt{L})f(z)|d\mu(z) \\
	&\lesi  \int_X \f{1}{V(z,t)} \Big(1+\f{d(x,z)}{t}\Big)^{-\lambda}|\varphi(t\sqrt{L})f(z)|d\mu(z)\\
	&\lesi \int_X  \f{1}{V(z,t)} \Big(1+\f{d(x,z)}{s}\Big)^{-\lambda}|\varphi(t\sqrt{L})f(z)|d\mu(z)\\
	&\lesi [M^*_{L,\varphi, \lambda}f(x)]^{1-\theta} \int_X \f{1}{V(z,t)}\Big(1+\f{d(x,z)}{t}\Big)^{-\theta \lambda} |\varphi(t\sqrt{L})f(z)|^\theta d\mu(z).
	\end{aligned}
	\end{equation}
	Applying Lemma \ref{lem-ele est} we have
	\[
	\begin{aligned}
	\Big(1+\f{d(x,y)}{t}\Big)^{-\lambda}|I_1(y,t)|
	&\lesi [M^*_{L,\varphi, \lambda}f(x)]^{1-\theta}  \mathcal M(|\varphi_L^+f|^\theta)(x).
	\end{aligned}
	\]
	Likewise, we have
	\[
	\begin{aligned}
	\Big(1+\f{d(x,y)}{t}\Big)^{-\lambda}|I_2(y,t)|
		&\lesi [M^*_{L,\varphi, \lambda}f(x)]^{1-\theta}\mathcal M(|\varphi_L^+f|^\theta)(x)..
	\end{aligned}
	\]
	Therefore, for all $y\in X$ and $0<t<d_X$ we have
	\[
	\begin{aligned}
	\Big(1+\f{d(x,y)}{t}\Big)^{-\lambda}|\varphi(t\sqrt{L})f(y)|
	&\lesi [M^*_{L,\varphi, \lambda}f(x)]^{1-\theta}\mathcal M(|\varphi_L^+f|^\theta)(x)
	\end{aligned}
	\]
	which implies that
	\[
	M^*_{L,\varphi, \lambda}f(x)\lesi [M^*_{L,\varphi, \lambda}f(x)]^{1-\theta}\mathcal M(|\varphi_L^+f|^\theta)(x).
	\]
As a result, we come up with
		\[
	M^*_{L,\varphi,N}f(x)\lesi \left[\mathcal{M}(|\varphi^+_{L}f|^{\theta})(x)\right]^{\f{1}{\theta}}.
	\]
	Since $p/\theta> q_w$ we have $w\in A_{p/\theta}$. Hence,
	\[
	\begin{aligned}
	\|M^*_{L,\varphi,N}f\|_{L^p_w(X)}&\lesi \left\|\left[\mathcal{M}(|\varphi^+_{L}f|^{\theta})(x)\right]^{\f{1}{\theta}}\right\|_{L^p_w(X)}\\
	&\lesi \|\varphi^+_{L}f\|_{L^p_w(X)}.
	\end{aligned}
	\]
	This completes our proof.
\end{proof}

\subsection{Maximal function characterizations}
We are ready to give the proof of Theorem \ref{maximal function result}.

\begin{proof}[Proof of Theorem \ref{maximal function result}:] The unweighted case of the theorem was proved in \cite{BDK2}, we now adapt this argument to our present situation with some modifications due to the presence of the weight $w\in A_\vc(X)$. We will give the proof for the case $\mu(X)<\vc$, since the case $\mu(X)=\vc$ is similar and even easier.
	
We now divide the proof into 2 steps.  

\textbf{Step 1.} Let $p\in (0,1]$, $w\in A_\vc(X)$  and $M>\f{n}{2}\big(\f{q_w}{p}-1\big)$. We now claim that $H^{p}_{L, w,{\rm max}}(X)\hookrightarrow H^{p,\vc}_{L,w,at,M}(X)$. To do this, fix $f\in H^{p}_{L, w,{\rm max}}(X)\cap L^2(X)$. We shall show that $f$ has an $(L,p,\vc,w,M)$-representation $\sum_j \lambda_ja_j$ with $(\sum_j|\lambda|_j^p)^{1/p} \lesssim \| f\|_{H^p_{L,w,\max}}$. 

Let $\Phi$ be a function from Lemma \ref{lem:finite propagation}. For $M\in \mathbb{N}, M>\f{n}{2}(\f{q_w}{p}-1)$ we have
\begin{equation}
\label{Calderon forula}
\begin{aligned}
f=c_{\Phi,M}\int_0^\vc (t^2L)^{M}\Phi(t\sqrt{L})\Phi(t\sqrt{L})f\f{dt}{t}
\end{aligned}
\end{equation}
in $L^2(X)$, where $\displaystyle c_{\Phi,M}= \Big[\int_0^\vc x^{2M}\Phi(x)\f{dx}{x}\Big]^{-1}$.	

Then we have 
\begin{equation}
\label{eq-f1 f2}
\begin{aligned}
f&=c_{\Phi,M}\int_0^{T_0} (t^2L)^{M}\Phi(t\sqrt{L})\Phi(t\sqrt{L})f\f{dt}{t}+c_{\Phi,M}\int^\vc_{T_0} (t^2L)^{M}\Phi(t\sqrt{L})\Phi(t\sqrt{L})f\f{dt}{t}\\
&=c_{\Phi,M}\int_0^{T_0} (t^2L)^{M}\Phi(t\sqrt{L})\Phi(t\sqrt{L})f\f{dt}{t}+\psi(T_0\sqrt{L})f\\
&=: f_1+f_2
\end{aligned}
\end{equation}
in $L^2(X)$ where $T_0=d_X/2$ and
\begin{equation}
\label{psi-eq}
\psi(x)=c_{\Phi,M}\int_x^\vc t^{2M}\Phi^2(t)\f{dt}{t}=c_{\Phi,M}\int_1^\vc (tx)^{2M}\Phi^2(tx)\f{dt}{t}.
\end{equation}
It is easy to check that   $\psi \in \mathscr{S}(\mathbb{R})$ and is an even function with $\psi(0)=1$. We now define the maximal operator
$$
\mathbb{M}_{L}f(x)=\sup_{0<t<d_X}\sup_{d(x,y)<8t}\left[ |\psi(t\sqrt{L})f(y)| +|\Phi(t\sqrt{L})f(y)|\right].
$$
Then Proposition \ref{prop2} yields
\begin{equation}
\label{eq1-proff mainthm1}
\|\mathbb{M}_{L}f\|_{L_w^p(X)}\lesi \|f\|_{H^p_{L,w\max}(X)}.
\end{equation}

Since $T_0={\rm diam}\,X/2$, we have, for any $x\in X$,
\[
|\psi(T_0\sqrt{L})f(x)|\leq \sup_{d(z,y)<8T_0}|\psi(T_0\sqrt{L})f(z)|\le \inf_{y\in X}\mathbb{M}_{L}f(y).
\]
This, along with \eqref{eq1-proff mainthm1}, implies that
\[
\begin{aligned}
\|f_2\|_{L^\vc(X)}&\le  w(X)^{-1/p}\|\mathbb{M}_{L}f\|_{L_w^p(X)}
\lesi  w(X)^{-1/p}\|f\|_{H^p_{L, w,\max}(X)}.
\end{aligned}
\]

Therefore, $f_2$ is an $(L,p,\vc,w,M)$ atom (with a harmless multiple constant).

We now take care of the component  $f_1$. For each $k\in \mathbb{Z}$ we set
$$
\Om_i:=\{x\in X: \mathbb{M}_{L}f(x)>2^i\}.
$$
Since $\mathbb{M}_{L}f$ is lower--continuous and $X$ is bounded, there exists $i_0$ so that $\Om_{i_0}=X$ and $\Om_{i_0+1}\neq X$. Without loss of generality we may assume that $i_0=0$. Then for each $t>0$ we define
\begin{equation}
\label{eq-Om it}
\Om_i^t=\begin{cases}
\Om_0, \ \ &i=0,\\
\{x: d(x,\Om_{i}^c)>4t\}, \ \ & i>0,
\end{cases}
\end{equation}
and $\widehat{\Omega}_i^t=\Om_i^t\backslash \Om_{i+1}^t$.

It is clear that $X=\bigcup_{i=0}^\vc \widehat{\Omega}_i^t$ for each $t>0$. Hence, 
\begin{equation}\label{eq1-f1}
\begin{aligned}
f_1&=\sum_{i=0}^\infty c_{\Phi,M}\int_0^{T_0} (t^2L)^{M}\Phi(t\sqrt{L})\left[\Phi(t\sqrt{L})f\cdot \chi_{\widehat{\Omega}_i^t}\right]\f{dt}{t}\\
&=:\sum_{i=0}^\infty f_1^i.
\end{aligned}
\end{equation}

Arguing similarly to \cite{BDK2} we obtain 
	\[
	|f_1^0(x)|\lesi 1, \ \ \ \forall x\in X
	\]
which implies that
\[
\begin{aligned}
|f_1^0(x)|&\lesi w(X)^{-1/p}w(\Om_0)^{1/p}\\
&\lesi w(X)^{-1/p}\sum_{i=0}^\vc 2^iw(\Om_i)^{1/p}\sim w(X)^{-1/p}\|\mathbb{M}_Lf\|_{L^p_w(X)}\\
&\lesi w(X)^{-1/p}\|f\|_{H^p_{L,w,\max}(X)}.	
\end{aligned}
\]
Hence, $f^1_0$ is an $(L,p,\vc,w,M)$ atom (with a harmless multiple constant).	 

We now take care of the term $f_1^i$ with $i>0$. To do this, for each $i>0$ we apply a covering lemma in \cite{CW} (see also \cite[Lemma 5.5]{DKP}) to obtain a collection of balls $\{B_{i,k}:=B(x_{B_{i,k}},r_{B_{i,k}}): x_{B_{i,k}}\in \Om_i, r_{B_{i,k}}=d(x_{B_{i,k}},\Om_i^c)/2, k=1,\ldots\}$ so that
\begin{enumerate}[{\rm (i)}]
	\item $\displaystyle \Om_i=\cup_k B(x_{B_{i,k}},r_{B_{i,k}})$;
	\item $\displaystyle  \{B(x_{B_{i,k}},r_{B_{i,k}}/5)\}_{k=1}^\vc$ are disjoint.
\end{enumerate}
For each $i, k\in \mathbb{N}^+$  and $t>0$ we set $B^t_{i,k}=B(x_{i,k},r_{B_{i,k}}+2t)$ which is a ball having the same center as $B_{i,k}$ with radius being $2t$ greater than the radius of  $B_{i,k}$. Then, for each $i, k\in \mathbb{N}^+$  and $t>0$, we set
\[
R_{i,k}^t=\begin{cases}
\widehat{\Omega}_i^t\cap B^t_{i,k}, \ \ &\text{if} \ \ \widehat{\Omega}_i^t\cap B_{i,k}\ne \emptyset\\
0,  \ \ &\text{if} \ \ \widehat{\Omega}_i^t\cap B_{i,k}= \emptyset,
\end{cases}
\]
and 
\begin{equation}
\label{eq-Eit}
E_{i,k}^t=R_{i,k}^t\backslash \cup_{\ell>k}R_{i,k}^t.
\end{equation}
It is easy to see that for each $i\in \mathbb{N}^+$ and $t>0$ we have
\[
\widehat{\Omega}_i^t =\bigcup_{k\in \mathbb{N}^+}E_{i,k}^t.
\]
Hence, from \eqref{eq1-f1} we have, for $i\in \mathbb{N}^+$,
\[
\begin{aligned}
f^i_1&=\sum_{k\in \mathbb{N}^+} c_{\Phi,M}\int_0^{T_0} (t^2L)^{M}\Phi(t\sqrt{L})\left[\Phi(t\sqrt{L})f\cdot \chi_{E_{i,k}^t}\right]\f{dt}{t}
\end{aligned}
\]
and set $a_{i,k}=0$ if $E_{i,k}^t=\emptyset$.

We now define $\lambda_{i,k}=2^i w(B_{i,k})^{1/p}$ and $a_{i,k}=L^Mb_{i,k}$ where
\begin{equation}\label{eq-bik}
b_{i,k}=\f{c_{\Phi,M}}{\lambda_{i,k}}\int_0^{T_0} t^{2M}\Phi(t\sqrt{L})\left[\Phi(t\sqrt{L})f\cdot \chi_{E_{i,k}^t}\right]\f{dt}{t}.
\end{equation}
Then it can be seen that 
$$
\begin{aligned}
f_1&=\sum_{i\in \mathbb{N}^+}f_1^i=\sum_{i,k\in \mathbb{N}^+}\lambda_{i,k}a_{i,k}
\end{aligned}
$$
in $L^2(X)$; moreover,
\[
\begin{aligned}
\sum_{i,k\in \mathbb{N}^+}|\lambda_{i,k}|^p&=\sum_{i,k\in \mathbb{N}^+} 2^{ip}w(B_{i,k})\lesi \sum_{i\in \mathbb{N}^+} 2^{ip}w(\Om_i)
\lesi \|\mathbb{M}_Lf\|^p_{L_w^p(X)}\lesi \|f\|^p_{H^p_{L,w,\max}(X)}.
\end{aligned}
\]
Therefore, it suffices to prove that each $a_{i,k}\ne 0$ is an $(L, p,\vc, w, M)$ atom associated to the ball $B^*_{i,k}:=8B_{i,k}$. Indeed, if $r_{B_{i,k}}<t/2$, then we have $d(x_{B_{i,k}}, \Om_i^c)=2r_{B_{i,k}}<t$. Therefore, 
$$
B^t_{i,k}=B(x_{B_{i,k}}, r_{B_{i,k}}+2t) \subset \{x: d(x,\Om_i^c)<4t\}.
$$
This implies that $R_{i,k}^t:=\widehat{\Omega}_i^t\cap B^t_{i,k}=\emptyset$. Hence, if $a_{i,k}\ne 0$, then $r_{B_{i,k}}\ge t/2$. This, along with \eqref{eq-bik} and Lemma \ref{lem:finite propagation}, implies that 
\[
\supp L^mb_{i,k}\subset B^*_{i,k}, \ \ \ \forall m=0,1,\ldots, M.
\]
By a similar argument to that in \cite{BDK2} we can show that
\[
|L^mb_{i,k}|_{L^\vc(X)}\lesi r_{B_{i,k}}^{2(M-m)}\mu(B_{i,k})^{-1/p}, \ \ \ \forall m=0,1,\ldots, M.
\]
This implies that each $a_{i,k}\ne 0$ is an $(L, p,\vc, w, M)$ atom and hence this completes the proof of Step 1.

\bigskip

\textbf{Step 2.} Let $p\in (0,1]$, $w\in A_\vc(X)$, $q\in (q_w,\vc]$, and $M>\f{n}{2}\big(\f{q_w}{p}-1\big)$. We now claim that $H^{p,q}_{L, w,{\rm max}}(X)\hookrightarrow H^{p}_{L, w,{\rm max}}(X)$. It suffices to show that there exists $C>0$ so that 
\begin{equation}
\label{eq-atom in Hp}
\left\|\sup_{0<t<d_X^2} |e^{-tL}a|\right\|_{L^p_w(X)}\le C
\end{equation}
for all $(L,p,q,w,M)$ atoms.

The proof of \eqref{eq-atom in Hp} is quite standard and is not difficult, hence we omit the details.

From the results in Steps 1 and 2 we conclude that for all $p\in (0,1]$, $w\in A_\vc(X)$, $q\in (q_w,\vc]$, and $M>\f{n}{2}\big(\f{q_w}{p}-1\big)$ we have
\[
H^{p,q}_{L, w,{\rm max}}(X)\equiv H^{p}_{L, w,{\rm max}}(X)
\]
which, together with Proposition \ref{prop2}, yields that
\[
H^{p,q}_{L, w,{\rm max}}(X)\equiv H^{p}_{L, w,{\rm max}}(X)\equiv H^{p}_{L, w,{\rm rad}}(X).
\]
This completes our proof.
\end{proof}    

\section{Weighted regularity estimates for inhomogeneous Dirichlet and Neumann problems}
  This section is dedicated to the proofs of Theorems \ref{mainthm1-bounded domain}-\ref{mainthm3-upper half-space}. To do this, we first prove a number of estimates for inhomogeneous Dirichlet problems and Neumann problems (see Subsections 3.1 and 3.2).
  These results are of  independent interest and should have applications in other settings apart from those in this paper. 
\subsection{Dirichlet Laplacian problems}

Let $\Om$ be a open connected domain in $\Rn$. Denote by $\Delta_D$ the Dirichlet Laplacian define on $\mathcal D(\Delta_D):=\{u\in W^{1,2}_0(\Om): \Delta u\in L^2(\Om)\}$ such that
\begin{equation}\label{eq-formula L}
\langle \Delta_Df, g \rangle =\int_{\Omega} \nabla f\cdot \nabla g, \ \ \ \forall f\in \mathcal D(\Delta_D), \ \ g\in W^{1,2}_0(\Om).
\end{equation}

Denote by $p_{t,\Delta_D}(x,y)$ the kernel of $e^{-t\Delta_D}$. It is well-known that 
\begin{equation}
\label{GU-Dirichlet}
0\le p_{t,\Delta_D}(x,y)\le \f{1}{(4\pi t)^{n/2}}\exp\Big(-\f{|x-y|^2}{4t}\Big)
\end{equation}
for all $t>0$ and $x,y\in \Omega$.

We now consider the following condition: There exist a constant $\beta>0$ and $C>0$ so that
\begin{equation}
\label{L2 gradient estimate}
\Big(\int_{x\in \Om} |\nabla_x^2 p_{t,\Delta_D}(x,y)|^2e^{\beta \f{|x-y|^2}{t}}dx\Big)^{1/2} \le C_\beta t^{-2}|B_\Om(y,\sqrt{t})|^{-1}, \ \ \ \forall y\in \Om, t>0.
\end{equation}

It is obviously that if  $\nabla_x^2 p_{t,\Delta_D}$ has a Gaussian upper bound, then \eqref{L2 gradient estimate} is satisfied. In our applications, 
we will show that \eqref{L2 gradient estimate} is satisfied if $\Om$ is one of the following domains:
\begin{enumerate}[(i)]
	\item a bounded, simply connected, semiconvex domain;
	\item a convex domain above a Lipschitz graph;
	\item the upper-half space.
\end{enumerate}
\begin{thm}\label{thm1-Dirichlet}
	Suppose that the second derivative of the Green function $\nabla^2 \Delta_D^{-1}$ is bounded on $L^{p_0}(X)$ for some $p_0\ge 2$. Suppose that the Dirichlet Laplacian $\Delta_D$ satisfies \eqref{L2 gradient estimate}. Then we have
	\begin{enumerate}[{\rm (i)}]
		\item The second derivative of the Green function $\nabla^2 \Delta_D^{-1}$ is bounded on $L^p_w(\Om)$ for all $1<p<p_0$ and $w\in A_p(\Rn)\cap RH_{(p_0/p)'}$, and is bounded from $L^1_w(\Om)$ into $L^{1,\vc}_w(\Om)$ for $w\in A_1\cap RH_{p_0'}$.
		\item The second derivative of the Green function $\nabla^2 \Delta_D^{-1}$ is bounded from $H^p_{\Delta_D,w}(\Om)$ into $H^p_{Mi,w}(\Om)$ for all $0<p\le 1$ and $w\in \bigcup_{1<r<p_0}A_r(\Rn)\cap RH_{(p_0/r)'}(\Rn)$.
	\end{enumerate}
\end{thm}

In order to prove the item (i) in Theorem \ref{thm1-Dirichlet} we need the following criterion for the weighted estimates for singular integrals whose proof is similar to that of \cite[Theorem 3.1]{BD}, and hence we omit the details.
\begin{thm}\label{Maintheorem-singularIntegral}
	Let $T$ be a bounded linear operator on $L^{p_0}(\Om)$ with $1<p_0<\vc$.  Also assume that there exist
	$m\in \mathbb{N}$, $\delta >0$ and $1 < p_2 < \vc$ such that for any ball $B\subset \Om$, the operator $T(I-e^{r_B^2L})^m$ has a kernel $K_{m,r_B}(x,y)$ satisfying
	\begin{equation}\label{eq1-maintheorem}
	\Big(\int_{S_j(B_\Om)}|K_{m,r_B}(y,z)|^{p_0}dy\Big)^{1/p_0}\leq C 2^{-j\delta}|2^jB_\Om|^{1/p_0-1}
	\end{equation}
	for all $z\in B$ and all $j\geq 2$.
	
Then, we have:
	
	(a) If $\delta>0$, then for any $1<p<p_0$ and $w\in A_p(\Rn)\cap RH_{(p_0/p)'}(\Rn)$, the operator $T$ is bounded on $L^p(\Om,w)$.
	
	(b) For any $w\in A_1(\Rn)\cap RH_{p'_0}(\Rn)$, if $\delta>n$, then $T$ is bounded from
	$L^1(\Om,w)$ into $L^{1,\infty}(\Om,w)$ .
\end{thm}

\begin{proof}[Proof of Theorem \ref{thm1-Dirichlet}:]  	
	(i) Fix $w\in A_p(\Rn)\cap RH_{(p_0/p)'}$. Then we can find $q_0\in (1,p_0)$ so that $w\in A_p(\Rn)\cap RH_{(q_0/p)'}(\Rn)$.  Fix a ball $B$ with radius $r_B$. For $m>n/2$, we observe that
	\begin{equation*}
	\Delta_D^{-1}=\int_{0}^{\infty}e^{-t\Delta_D}\,dt
	\end{equation*}
	so that
	\begin{equation}\label{Kmr}
	\nabla^2\Delta_D^{-1}(I-e^{r_B^2\Delta_D})^m=\int_{0}^{\infty}\nabla^2e^{-t\Delta_D}(I-e^{-r_B^2\Delta_D})^m\frac{dt}{\sqrt{t}}=\int_{0}^{\infty}g_{r_B,m}(t)\nabla^2e^{-t\Delta_D}dt
	\end{equation}
	where $g_{r_B,m}:\mathbb{R}^+\rightarrow \mathbb{R}$ is a function such that
	\begin{equation}\label{g-function}
	\int_{0}^{\infty}|g_{r_B,m}(t)|e^{\frac{-\alpha 4^jr_B^2}{t}} \,\f{dt}{t}
	\leq C_{m,\alpha}4^{-jm}, 
	\end{equation}
	for any $\alpha>0$. See for example \cite[p.932]{ACDH}.\\
	
	It follows that
	\[
	K_{m,r_B}(y,z)=\int_{0}^{\infty}g_{r_B,m}(t)\nabla^2p_{t,\Delta_D}(y,z)dt.
	\]
	Hence, for $z\in B$ and $j\ge 2$ we have
	\[
	\Big(\int_{S_j(B_\Om)}|K_{m,r_B}(y,z)|^{q_0}dy\Big)^{1/q_0}\le 	\int_{0}^{\infty}g_{r_B,m}(t)\Big(\int_{S_j(B_\Om)}|\nabla^2p_{t,\Delta_D}(y,z)|^{q_0}dy\Big)^{1/q_0}dt.
	\]
	
	On the other hand, since $\nabla^2 \Delta_D^{-1}$ is bounded on $L^{p_0}$
	\begin{equation}
	\label{Lp0 gradient estimate}
	\begin{aligned}
	\Big(\int_{S_j(B)} |\nabla_y^2 p_{t,\Delta_D}(y,z)|^{p_0}dy\Big)^{1/p_0} &\le \Big(\int_{\Om} |\nabla_y^2 p_{t,\Delta_D}(y,z)|^{p_0}dy\Big)^{1/p_0}\\
	 &\le \Big(\int_{\Om} |\Delta_D p_{t,\Delta_D}(y,z)|^{p_0}dy\Big)^{1/p_0}\\
	&\lesi t^{\f{n}{2p_0}-\f{n}{2}-1}\\
	&\lesi |2^jB_\Om|^{1/p_0-1}t^{-1} \Big(\f{4^jr_B^2}{t}\Big)^{\f{n}{2}-\f{n}{2p_0}}
	\end{aligned}
	\end{equation}
for all   $z\in B_\Om$ and $t>0$.
    From \eqref{L2 gradient estimate}, we have
    \[
    \begin{aligned}
    \Big(\int_{S_j(B_\Om)}|\nabla^2p_{t,\Delta_D}(y,z)|^2dy\Big)^{1/2}&\lesi | B_\Om(z,\sqrt{t})|^{-1/2}t^{-1}e^{-\beta\f{4^{j}r^2_B}{t}}\\
    &\lesi | 2^jB_\Om|^{-1/2}t^{-1}e^{-\beta\f{4^{j}r^2_B}{2t}}.
    \end{aligned}
    \]
Interpolating this and \eqref{Lp0 gradient estimate} we obtain
    \[
    \Big(\int_{S_j(B_\Om)}|\nabla^2p_{t,\Delta_D}(y,z)|^{q_0}dy\Big)^{1/q_0}\lesi |2^jB_\Om|^{1/q_0-1}t^{-1}e^{-\beta'\f{4^{j}r^2_B}{t}}.
    \]
	
	Therefore,
	\[
	\Big(\int_{S_j(B_\Om)}|K_{m,r_B}(y,z)|^{q_0}dy\Big)^{1/q_0}\lesi 	|2^jB_\Om|^{1/q_0-1}\int_{0}^{\infty}g_{r_B,m}(t) e^{-\beta\f{4^{j}r_B^2}{t}}  \f{dt}{t}
	\]
	which along with \eqref{g-function} implies that
	\[
	\Big(\int_{S_j(B_\Om)}|K_{m,r_B}(y,z)|^{q_0}dy\Big)^{1/q_0}\lesi 2^{-2mj} 	|2^jB_\Om|^{1/q_0-1}, \ \ \ j\ge 2.
	\]
	Applying Theorem \ref{Maintheorem-singularIntegral}, we get (i).
	
	\bigskip
	
	(ii) Fix $w\in \bigcup_{1<r<p_0}A_r(\Rn)\cap RH_{(p_0/r)'}(\Rn)$. Then there exists $q\in (1,p_0)$ so that $w\in A_q\cap RH_{(p_0/q)'}$. Let $a=\Delta_D^M b$ be an $(\Delta_D, p, q,w,M)$ atom associated to a ball $B$ where $M\in \mathbb{N}, \ M>\f{1}{2}\lfloor n(q_w/p-1)\rfloor$. We claim that $\nabla^2 \Delta^{-1}_D a$ is a $(p,q,w)_{Mi}$ atom. Indeed, we consider three cases.
	
	\bigskip
	
	\textbf{Case 1: $4B\subset \Om$.} We first observe that $\nabla^2 \Delta^{-1}_D a=\nabla^2 \Delta^{M-1}_D b$. Since supp\,$\Delta^{M-1}_D b\subset B$, we have supp\,$\nabla^2 \Delta^{-1}_D a\subset B$. Moreover, from (i) we have
	\[
	\|\nabla^2 \Delta^{-1}_D a\|_{L^q_w(\Om)}\lesi \|a\|_{L^q_w(\Om)}\lesi w(B)^{1/q-1/p}.
	\]
	We now verify that 
	\begin{equation}\label{eq-cancelation}
	\int_B x^\alpha \nabla^2 \Delta^{-1}_D a(x) dx =0
	\end{equation}
	for every multi-index $\alpha$ with $|\alpha|\le \lfloor n(q_w/p-1)\rfloor$.
	
	Let $\psi\in C^\vc(\Om)$ so that supp\,$\psi\subset 2B$ and $\psi=1$ in $\f{3}{2}B$. We then have supp\,$\partial^\beta [x^\alpha\psi(x)]\subset 2B$ and $\partial^\beta [x^\alpha\psi(x)]=0$ on $(2B)^c$ for every multi-index $\beta$. Therefore, by integration by part we have
	\[
	\begin{aligned}
	\int_B x^\alpha \nabla^2 \Delta^{-1}_D a(x) dx&= \int_B x^\alpha\psi(x) \nabla^2 \Delta^{M-1}_D b(x) dx\\
	&= \int_B \nabla^2 [x^\alpha\psi(x)] \Delta^{M-1}_D b(x) dx . \\
	\end{aligned}
	\]
	Since $\nabla^2 [x^\alpha\psi(x)]\in W^{1,2}_0(\Om)$, we have
	\[
	\begin{aligned}
	\int_B x^\alpha \nabla^2 \Delta^{-1}_D a(x) dx
	&= \int_B \nabla^3 [x^\alpha\psi(x)] \nabla \Delta^{M-2}_D b(x) dx.
	\end{aligned}
	\]
	Using integration by part again, we have
	\[
	\begin{aligned}
	\int_B x^\alpha \nabla^2 \Delta^{-1}_D a(x) dx
	&= \int_B \nabla^4 [x^\alpha\psi(x)] \Delta^{M-2}_D b(x) dx.
	\end{aligned}
	\]
	Repeating this process $(M-2)$ times we come up with 
	\[
	\begin{aligned}
	\int_B x^\alpha \nabla^2 \Delta^{-1}_D a(x) dx
	&= \int_B \nabla^{2M} [x^\alpha\psi(x)] b(x) dx.
	\end{aligned}
	\] 
	Since $\psi=1$ on $B$ and $M>\f{1}{2}\lfloor n(q_w/p-1)\rfloor$, we have $\nabla^{2M} [x^\alpha\psi(x)]=0$ for $x\in B$. As a consequence,
	\[
	\begin{aligned}
	\int_B x^\alpha \nabla^2 \Delta^{-1}_D a(x) dx
	&= 0
	\end{aligned}
	\]
	which proves \eqref{eq-cancelation}. 
	
	Therefore, $\nabla^2 \Delta^{-1}_D a$ is a $(p,q,w)_{Mi}$ atom associated to $B$.
	
	\bigskip
	
	\textbf{Case 2: $4B\cap \Om^c\ne \emptyset$.}  Set $\tilde a=\nabla^2\Delta_D^{-1} a$. Arguing similar to Case 1, we have supp\,$\tilde a \subset B_\Om$ and 
	\[
	\|\tilde a\|_{L^q_w(\Om)}\lesi w(B_\Om)^{1/q-1/p}\sim w(B)^{1/q-1/p}. 
	\]
	From Theorem \ref{MiyachiTheorem} we need to claim that
	\[
	\|\tilde a^+_\Om\|_{L^p_w(\Om)}\lesi 1.
	\]
	Indeed, let $\phi$ be a function as in \eqref{eq-phi}. We have
	\[
	\|\tilde a^+_\Om\|_{L^p_w(\Om)}\lesi \|\tilde a^+_\Om\|_{L^p_w(8B_\Om)} +\|\tilde a^+_\Om\|_{L^p_w((8B_\Om)^c)}.
	\]
	We claim that $\|\tilde a^+_\Om\|_{L^p_w((8B_\Om)^c)}=0$. Indeed, for $x\in (8B_\Om)^c$ we have
	\[
	\begin{aligned}
	\tilde a^+_\Om(x) &=\sup_{0<t<\delta(x)/2}\Big|\int_{B_\Om} \f{1}{t^n}\phi\Big(\f{x-y}{t}\Big)\tilde a(y)dy\Big|\\
	&\le \sup_{0<t<7r_B}\Big|\int_{B_\Om} \f{1}{t^n}\phi\Big(\f{x-y}{t}\Big)\tilde a(y)dy\Big|+\sup_{7r_B\le t<\delta(x)/2}\Big|\int_{B_\Om} \f{1}{t^n}\phi\Big(\f{x-y}{t}\Big)\tilde a(y)dy\Big|\\
	&=: E_1 + E_2.
	\end{aligned}
	\]
	Since $|x-y|>7r_B$ for $y\in B_\Om$ and $x\in (8B_\Om)^c$, $E_1=0$. Note that the term $E_2$ is valid if $\delta(x)>14 r_B$. In this situation, for $y\in B_\Om$ and $x\in (8B_\Om)^c$ we have
	\[
	|x-y|>\delta(x)-\delta(y).
	\]
	Since $4B\cap \Om^c\ne \emptyset$, we have $\delta(y)<4r_B$ for  each $y\in B_\Om$. Hence, for $y\in B_\Om$ and $x\in (8B_\Om)^c$ we have
	\[
	|x-y|>\delta(x)-\delta(y)>\delta(x)-4r_B>\delta(x)-\delta(x)/2=\delta(x)/2>t.
	\]
	As a consequence, 
	\[
	\phi\Big(\f{x-y}{t}\Big)=0
	\]
	for $y\in B_\Om$ and $x\in (8B_\Om)^c$ and $7r_B\le t<\delta(x)/2$. Hence, $E_2=0$.
	
	The estimates of $E_1$ and $E_2$ yield $\|\tilde a^+_\Om\|_{L^p_w((8B_\Om)^c)}=0$.
	
	Therefore,  
	\[
	\|\tilde a^+_\Om\|_{L^p_w(\Om)}\lesi \|\tilde a^+_\Om\|_{L^p_w(8B_\Om)}.
	\]
	Applying H\"older's inequality we obtain
	\[
	\begin{aligned}
	\|\tilde a^+_\Om\|_{L^p_w(8B_\Om)}&\lesi \|\tilde a^+_\Om\|_{L^q_w(8B_\Om)}w(8B_\Om)^{1/p-1/q}\\
	&\lesi w(B_\Om)^{1/q-1/p}w(8B_\Om)^{1/p-1/q}\\
	&\lesi 1
	\end{aligned}
	\]
	where in the second inequality we used the fact that $\tilde a^+_\Om \lesi \mathcal{M}\tilde a$ and $\mathcal{M}$ is bounded on $L^q_w$ for $w\in A_q$.
	\bigskip
	
	\textbf{Case 3: $\Om$ is bounded.} In this case, apart from the atoms considered in two cases above, it remains to consider the case $a=w(\Om)^{-1/p}\chi_\Om$. This case can be done similarly to that of the case 2.

	This completes our proof.
\end{proof}
The next result gives regularity estimates for the Dirichlet Green operator.
\begin{thm}\label{thm2-Dirichlet}
	Suppose that the second derivative of the Green function $\nabla^2 \Delta_D^{-1}$ is bounded on $L^{p_0}(X)$ for some $p_0\ge 2$. Assume that the Dirichlet Laplacian $\Delta_D$ satisfies \eqref{L2 gradient estimate}, and  there exist $\gamma\in (0,1]$ and $C,c>0$ so that
	\begin{equation}\label{Holder-Dirichlet}
	|p_{t,\Delta_D}(x,y)-p_{t,\Delta_D}(x,y')|\le \Big(\f{|y-y'|}{\sqrt{t}}\Big)^\gamma\f{C}{\mu(B_\Omega(x,\sqrt{t}))}\exp\Big(-\f{|x-y|^2}{ct}\Big)
	\end{equation} 
	for all $0<t<d_\Om:={\rm diam}\,\Om$ and  $x,y,y'\in \Om$ so that $|y-y'|<\sqrt{t}/2$. 
	
	Then the second derivative of the Green function $\nabla^2 \Delta_D^{-1}$ is bounded on $H^p_{Mi,w}(\Om)$ for all $\f{n}{n+\gamma}<p\le 1$ and $w\in \bigcup_{1<r<r_0}A_r(\Rn)\cap RH_{(p_0/r)'}(\Rn)$ where $r_0=\f{p(n+\gamma)}{n}$. 
	
	Hence, $\nabla^2 \Delta_D^{-1}$ is bounded from $H^p_{z,w}(\Om)$ into $H^p_{Mi,w}(\Om)$ for all $\f{n}{n+\gamma}<p\le 1$ and $w\in \bigcup_{1<r<r_0}A_r(\Rn)\cap RH_{(p_0/r)'}(\Rn)$.
\end{thm}
\begin{proof}
	From Theorem \ref{thm1-Dirichlet} in order to prove that $\nabla^2 \Delta_D^{-1}$ is bounded on $H^p_{Mi,w}(\Om)$,  it suffices to show that \[
	H^p_{Mi,w}(\Om)\hookrightarrow H^p_{\Delta_D, w}(\Om)
	\]
	for  all $\f{n}{n+1}<p\le 1$ and $w\in \bigcup_{1<r<p_0}A_r(\Rn)\cap RH_{(p_0/r)'}(\Rn)$.
	
	By Theorem \ref{maximal function result} it suffices to prove that there exists $C>0$ so that
	\[
	\|\mathcal M_{\Delta_D}a\|_{L^p_w(\Om)}\le C
	\]
	for every $(p, \vc,w)_{Mi}$-atom $a$ associated to a ball $B$, where
	\[
	\mathcal M_{\Delta_D}a(x)= \sup_{0<t<d_\Om^2}|e^{-t\Delta_D}a(x)|.
	\]
	Indeed, we have
	\[
	\|\mathcal M_{\Delta_D}a\|_{L^p_w(\Om)}\lesi \|\mathcal M_{\Delta_D}a\|_{L^p_w(4B)} +\|\mathcal M_{\Delta_D}a\|_{L^p_w(\Om\backslash 4B)}.
	\]
	Since $\mathcal M_{\Delta_D}a\lesi \mathcal{M}a$, we have
	\[
	\|\mathcal M_{\Delta_D}a\|_{L^p_w(4B)}\lesi w(4B)^{1/p}\|\mathcal M_{\Delta_D}a\|_{L^\vc}\lesi w(4B)^{1/p}\|a\|_{L^\vc}\lesi 1.
	\]
	For the second term we remark that
	\[
	\mathcal M_{\Delta_D}a(x)\le  \sup_{0<t<4r_B^2}|e^{-t\Delta_D}a(x)|+ \sup_{4r_B^2\le t<d_\Om^2}|e^{-t\Delta_D}a(x)|=:I_1(x)+I_2(x).
	\]
	By the Gaussian upper bound \eqref{GU-Dirichlet} we have
	\[
	\begin{aligned}
	I_1(x)&\lesi \sup_{0<t<4r_B^2}\int_B \f{1}{t^{n/2}}\exp\Big(-\f{|x-y|^2}{ct^2}\Big)|a(y)|dy\sim \sup_{0<t<4r_B^2}\int_B \f{1}{t^{n/2}}\exp\Big(-\f{|x-x_B|^2}{ct^2}\Big)|a(y)|dy\\
	&\lesi \|a\|_{L^1}\f{r_B^\gamma}{|x-x_B|^{n+\gamma}}.
	\end{aligned}
	\] 
	For the term $I_2(x)$ we consider two cases. If $4B\subset \Om$, then by the cancellation property $\int a = 0$ and \eqref{Holder-Dirichlet} we will come up with
	\[
	I_2(x)\lesi \|a\|_{L^1}\f{r_B^\gamma}{|x-x_B|^{n+\gamma}}.
	\]
	If $2B\subset \Om$ and $4B\cap \Om^c\ne \emptyset$, then we have
	\[
	\begin{aligned}
	I_2(x)&= \sup_{0<t<4r_B^2}\Big|\int_B \left(p_{t,\Delta_D}(x,y)-p_{t,\Delta_D}(x,y_0)\right) a(y)dy\Big|
	\end{aligned}
	\]
	where $y_0$ is any point in $4B\cap \partial\Om$.
	
	Using \eqref{Holder-Dirichlet} we also obtain
	\[
	I_2(x)\lesi \|a\|_{L^1}\f{r_B^\gamma}{|x-x_B|^{n+\gamma}}.
	\]
	
	Taking the estimates of $I_1$ and $I_2$ into account we find that
	\[
	\mathcal M_{\Delta_D}a(x)\lesi \|a\|_{L^1}\f{r_B^\gamma}{|x-x_B|^{n+\gamma}}\lesi \f{r_B^\gamma}{|x-x_B|^{n+\gamma}} |B|w(B)^{-1/p}.
	\] 
	This implies
	\[
	\|\mathcal M_{\Delta_D}a\|_{L^p_w(\Om\backslash 4B)}\lesi 1
	\]
	as long as $w\in \bigcup_{1<r<r_0}A_r(\Rn)\cap RH_{(p_0/r)'}(\Rn)$.
\end{proof}
We can obtain regularity estimates for the Dirichlet Green operator  for a larger range of $p$ if we have stronger assumptions on the derivatives of heat kernels.
\begin{thm}\label{thm3-Dirichlet}
	Suppose that the second derivative of the Green function $\nabla^2 \Delta_D^{-1}$ is bounded on $L^{p_0}(X)$ for some $p_0\ge 2$. Assume that the Dirichlet Laplacian $\Delta_D$ satisfies \eqref{L2 gradient estimate}, and assume that for any multi-index $\alpha$, there exist  $C, c>0$ so that
	\begin{equation}\label{n derivative Holder-Dirichlet}
	|\partial_x^\alpha p_{t,\Delta_D}(x,y)|\le \f{C}{t^{|\alpha|}|B_\Omega(x,\sqrt{t})|}\exp\Big(-\f{|x-y|^2}{ct}\Big)
	\end{equation}
	for all $0<t<{\rm diam}\,\Om$ and $x,y\in \Omega$. 
	
	Then the second derivative of the Green function $\nabla^2 \Delta_D^{-1}$ is bounded  $H^p_{z,w}(\Om)$ into $H^p_{Mi,w}(\Om)$ for all $0<p\le 1$ and $w\in \bigcup_{1<r<p_0}A_r(\Rn)\cap RH_{(p_0/r)'}(\Rn)$.
\end{thm}
\begin{proof}
By a similar argument to that of the proof of Theorem \ref{thm2-Dirichlet} with minor modifications, we can show that for all $0<p\le 1$ and $w\in \bigcup_{1<r<p_0}A_r(\Rn)\cap RH_{(p_0/r)'}(\Rn)$ there exists $C>0$ such that
\[
\|\mathcal M_{\Delta_D}a\|_{L^p_w(\Om)}\le C
\]
for every $(p, \vc,w)_{Mi}$-atom $a$ associated to a ball $B$.

This, along with Theorem \ref{maximal function result}, implies that $H^p_{z,w}(\Om) \subset H^p_{\Delta_D,w}(\Om)$. Hence, Theorem \ref{thm3-Dirichlet} follows immediately from Theorem \ref{thm1-Dirichlet}.
\end{proof}

Note that in Theorems \ref{thm1-Dirichlet}, \ref{thm2-Dirichlet} and \ref{thm3-Dirichlet},  condition \eqref{L2 gradient estimate} is only used to obtain the weighted estimates. Hence,  in the unweighted case when $w\equiv 1$, condition \eqref{L2 gradient estimate} can be removed. More precisely, we have the following result.  
\begin{thm}\label{thm1-Dirichlet-unweighted}
	Suppose the second derivative of the Green function $\nabla^2 \Delta_D^{-1}$ is bounded on $L^{p_0}(X)$ for some $p_0\ge 2$. Then we have:
	\begin{enumerate}[{\rm (i)}]
		\item The operator $\nabla^2 \Delta_D^{-1}$ is bounded from $H^p_{\Delta_D}(\Om)$ into $H^p_{Mi}(\Om)$ for all $0<p\le 1$.
		\item If \eqref{Holder-Dirichlet} is satisfied, then  $\nabla^2 \Delta_D^{-1}$ is bounded on $H^p_{Mi}(\Om)$ for all $\f{n}{n+\gamma}<p\le 1$. Hence, $\nabla^2 \Delta_D^{-1}$ is bounded from  $H^p_{z}(\Om)$ into $H^p_{Mi}(\Om)$ for all $\f{n}{n+\gamma}<p\le 1$.
		\item If \eqref{n derivative Holder-Dirichlet} is satisfied, then  $\nabla^2 \Delta_D^{-1}$ is bounded from  $H^p_{z}(\Om)$ into $H^p_{Mi}(\Om)$ for all $0<p\le 1$.
	\end{enumerate}
\end{thm}

\subsection{Neumann Laplacian problems}
Let $\Om$ be a open connected domain in $\Rn$. Denote by $\Delta_N$ the Neumann Laplacian define on $\mathcal D(\Delta_N):=\{u\in W^{1,2}(\Om): \Delta u\in L^2(\Om)\}$ such that
\begin{equation}\label{eq-formula LN}
\langle \Delta_Nf, g \rangle =\int_{\Omega} \nabla f\cdot \nabla g, \ \ \ \forall f\in \mathcal D(\Delta_N), \ \ g\in W^{1,2}(\Om).
\end{equation}

Denote by $p_{t,\Delta_N}(x,y)$ the kernel of $e^{-t\Delta_N}$. We assume that there exist $C,c>0$ so that  
\begin{equation}
\label{GU-Neumann}
0\le p_{t,\Delta_N}(x,y)\le \f{C}{|B_\Om(x,\sqrt{t})|}\exp\Big(-\f{|x-y|^2}{ct}\Big)
\end{equation}
for all $t>0$ and $x,y\in \Omega$, where $B_\Om(x,\sqrt{t})=B(x,\sqrt{t})\cap \Om$. It was proved in \cite{Da} that the Guassian upper bound \eqref{GU-Neumann} is satisfied if $\Om$ if Ω has the extension property.

We now consider the following condition: There exist a constant $\beta>0$ and $C>0$ so that
\begin{equation}\label{L2 gradient estimate-Neumann}
\Big(\int_{x\in \Om} |\nabla_x^2 p_{t,\Delta_N}(x,y)|^2e^{\beta \f{|x-y|^2}{t}}dx\Big)^{1/2} \le C_\beta t^{-2}|B_\Om(y,\sqrt{t})|^{-1}, \ \ \ \forall y\in \Om, t>0.
\end{equation}
Similarly to the condition \eqref{L2 gradient estimate}, it can be seen that the estimates \eqref{L2 gradient estimate-Neumann} is satisfied if  $\nabla_x^2 p_{t,\Delta_N}$ has a Gaussian upper bound. In our applications, we will show that the estimate \eqref{L2 gradient estimate-Neumann} is satisfied if $\Om$ is one of the following domains:
\begin{enumerate}[(i)]
	\item a bounded convex domain;
	\item a convex domain above a Lipschitz graph;
	\item the upper-half space.
\end{enumerate}

Arguing similarly to the proof of Theorem \ref{thm1-Dirichlet} we have:
\begin{thm}\label{thm1-Neumann}
	Suppose that the second derivative of the Green function $\nabla^2 \Delta_N^{-1}$ is bounded on $L^{p_0}(X)$ for some $p_0\ge 2$. Suppose that the Neumann Laplacian $\Delta_N$ satisfies \eqref{L2 gradient estimate-Neumann}. Then we have
	\begin{enumerate}[{\rm (i)}]
		\item The second derivative of the Green function $\nabla^2 \Delta_N^{-1}$ is bounded on $L^p_w(\Om)$ for all $1<p<p_0$ and $w\in A_p(\Rn)\cap RH_{(p_0/p)'}$, and is bounded from $L^1_w(\Om)$ into $L^{1,\vc}_w(\Om)$ for $w\in A_1\cap RH_{p_0'}$.
		\item The second derivative of the Green function $\nabla^2 \Delta_N^{-1}$ is bounded from $H^p_{\Delta_N,w}(\Om)$ into $H^p_{Mi,w}(\Om)$ for all $0<p\le 1$ and $w\in \bigcup_{1<r<p_0}A_r(\Rn)\cap RH_{(p_0/r)'}(\Rn)$.
	\end{enumerate}
\end{thm}

In the particular case when the kernel $p_{t,\Delta_N}(x,y)$ satisfies the H\"older continuity \eqref{Holder-Neumann} below, Theorem \ref{thm1-Neumann} deduces to the boundedness of $\nabla^2 \Delta_N^{-1}$ from $H^p_{z,w}(\Om)$ into $H^p_{Mi,w}(\Om)$ for certain $p\le 1$. We have:
\begin{thm}\label{thm2-Neumann}
	Suppose that the second derivative of the Green function $\nabla^2 \Delta_N^{-1}$ is bounded on $L^{p_0}(X)$ for some $p_0\ge 2$. Assume that the Neumann Laplacian $\Delta_N$ satisfies \eqref{L2 gradient estimate}, and  there exist $\gamma\in (0,1]$ and $C,c>0$ so that
	\begin{equation}\label{Holder-Neumann}
	|p_{t,\Delta_N}(x,y)-p_{t,\Delta_N}(x,y')|\le \Big(\f{|y-y'|}{\sqrt{t}}\Big)^\gamma\f{C}{\mu(B_\Omega(x,\sqrt{t}))}\exp\Big(-\f{|x-y|^2}{ct}\Big)
	\end{equation} 
	for all $0<t<d_\Om:={\rm diam}\,\Om$ and  $x,y,y'\in \Om$ so that $|y-y'|<\sqrt{t}/2$. 
	
	Then the second derivative of the Green function $\nabla^2 \Delta_N^{-1}$ is bounded from $H^p_{z,w}(\Om)$ into $H^p_{Mi,w}(\Om)$ for all $\f{n}{n+\gamma}<p\le 1$ and $w\in \bigcup_{1<r<r_0}A_r(\Rn)\cap RH_{(p_0/r)'}(\Rn)$ where $r_0=\f{p(n+\gamma)}{n}$. 
\end{thm}
\begin{proof}
	From Theorem \ref{thm1-Neumann} in order to prove that $\nabla^2 \Delta_N^{-1}$ is bounded from $H^p_{z,w}(\Om)$ into $H^p_{Mi,w}(\Om)$,  it suffices to show that \[
	H^p_{z,w}(\Om)\hookrightarrow H^p_{\Delta_N, w}(\Om)
	\]
	for  all $\f{n}{n+\gamma}<p\le 1$ and $w\in \bigcup_{1<r<p_0}A_r(\Rn)\cap RH_{(p_0/r)'}(\Rn)$.
	
	To do this, we note that it can be verified that there exists $C>0$ so that 
	\[
	\|\mathcal M_{\Delta_N}a\|_{L^p_w(\Om)}\le C
	\]
	for every $(p,q,w)_\Om$-atom $a$ associated to a ball $B$, where
	\[
	\mathcal M_{\Delta_N}a(x)= \sup_{0<t<d_\Om^2}|e^{-t\Delta_N}a(x)|.
	\]
	It follows that 
	\[
	H^p_{z,w}(\Om)\hookrightarrow H^p_{\Delta_N, w, \max}(\Om).
	\]
	This, in combination with  Theorem \ref{maximal function result}, yields 
	\[
	H^p_{z,w}(\Om)\hookrightarrow H^p_{\Delta_N, w}(\Om)
	\]
	which completes our proof.
\end{proof}

\begin{thm}\label{thm3-Neumann}
	Suppose that the second derivative of the Green function $\nabla^2 \Delta_N^{-1}$ is bounded on $L^{p_0}(X)$ for some $p_0\ge 2$. Assume that the Neumann Laplacian $\Delta_N$ satisfies \eqref{L2 gradient estimate}, and assume that for any multi-index $\alpha$, there exist  $C, c>0$ so that
	\begin{equation}\label{n derivative Holder-Neumann}
	|\partial_x^\alpha p_{t,\Delta_N}(x,y)|\le \f{C}{t^{|\alpha|}|B_\Omega(x,\sqrt{t})|}\exp\Big(-\f{|x-y|^2}{ct}\Big)
	\end{equation}
	for all $0<t<{\rm diam}\,\Om$ and $x,y\in \Omega$. 
	
	Then the second derivative of the Green function $\nabla^2 \Delta_N^{-1}$ is bounded  $H^p_{z,w}(\Om)$ into $H^p_{Mi,w}(\Om)$ for all $0<p\le 1$ and $w\in \bigcup_{1<r<p_0}A_r(\Rn)\cap RH_{(p_0/r)'}(\Rn)$.
\end{thm}
\begin{proof}
	From Theorem \ref{thm1-Neumann} it suffices to verify that
	\[
	H^p_{z,w}(\Om)\hookrightarrow H^p_{\Delta_N,w}(\Om) .
	\]
	To do this, we need to show that for all $0<p\le 1$ and $w\in \bigcup_{1<r<p_0}A_r(\Rn)\cap RH_{(p_0/r)'}(\Rn)$ there exists $C>0$ so that 
	\begin{equation}
	\label{eq-thm3 Neumann}
	\|\mathcal M_{\Delta_N}a\|_{L^p_w(\Om)}\le C
	\end{equation}
	for every $(p,q,w)_\Om$-atom $a$.
	
	The proof of \eqref{eq-thm3 Neumann} is standard and we omit the details. This completes our proof.
	
\end{proof}

It is worth noticing that in the unweighted case when $w\equiv 1$,  condition \eqref{L2 gradient estimate-Neumann} can be removed. We have the following result.  
\begin{thm}\label{thm1-Neumann-unweighted}
	Suppose the second derivative of the Green function $\nabla^2 \Delta_N^{-1}$ is bounded on $L^{p_0}(X)$ for some $p_0\ge 2$. Then we have:
	\begin{enumerate}[{\rm (i)}]
		\item The operator $\nabla^2 \Delta_N^{-1}$ is bounded from $H^p_{\Delta_N}(\Om)$ into $H^p_{Mi}(\Om)$ for all $0<p\le 1$.
		\item If \eqref{Holder-Neumann} is satisfied,  $\nabla^2 \Delta_N^{-1}$ is bounded from  $H^p_{z}(\Om)$ into $H^p_{Mi}(\Om)$ for all $\f{n}{n+\gamma}<p\le 1$.
		\item If \eqref{n derivative Holder-Neumann} is satisfied, then  $\nabla^2 \Delta_N^{-1}$ is bounded from  $H^p_{z}(\Om)$ into $H^p_{Mi}(\Om)$ for all $0<p\le 1$.
	\end{enumerate}
\end{thm}

\subsection{Proof of main results} We now ready to give the proofs of Theorems \ref{mainthm1-bounded domain}--\ref{mainthm3-upper half-space}.

\begin{proof}
	[Proof of Theorem \ref{mainthm1-bounded domain}:] Note that the $L^2$-boundedness of $\nabla^2 \Delta_D^{-1}$ and $\nabla^2 \Delta_N^{-1}$ can be found in \cite[Theorem 4.8]{DHMMY}.  The condition \eqref{Holder-Dirichlet} for any $\gamma\in (0,1)$ and  condition \ref{L2 gradient estimate} for the inhomogeneous Dirichlet problems were verified in Lemma 2.7 and Proposition 4.15 in \cite{DHMMY}. Meanwhile, the Gaussian upper bound \eqref{GU-Neumann}, the H\"older continuity condition \ref{Holder-Neumann} for $\gamma=1$ and the estimate \eqref{L2 gradient estimate-Neumann} can be found in  Lemma 2.8 and Proposition 4.15 in \cite{DHMMY}. Therefore, Theorem \ref{mainthm1-bounded domain} follows directly from Theorem \ref{thm1-Dirichlet}, Theorem \ref{thm2-Dirichlet}, Theorem \ref{thm1-Neumann} and Theorem \ref{thm2-Neumann}.
\end{proof}

\bigskip

To prove Theorem \ref{mainthm2-above convex Lipschitz domains}, we need the following technical result:
\begin{lem}\label{lemAJ}
	(a) Let $\Omega$  be a convex domain above a Lipschitz graphs in $\Rn$, and
	assume that $f \in L^2(\Omega)$. Then the unique solution $u \in W^{1,2}(\Omega)$ to the Neumann problem \eqref{NeumannProblem} has the property that for any  $ \psi \in C^\vc(\Rn)$,
	\begin{eqnarray}
	\int_\Omega \psi^2|\nabla^2 u|^2dx \leq C\int_\Omega |\nabla \psi|^2|\nabla u|^2 dx + C\int_\Omega \psi^2 f^2 dx,
	\end{eqnarray}
	for some finite constant $C>0$ independent of $f$.
	
	(b) The result of part (a) still holds if the function $f \in L^2(\Omega)$ with $\int_\Omega
	fdx=0$ and the Neumann problem \eqref{NeumannProblem} is replaced by the Dirichlet problem \eqref{NeumannProblem}.
\end{lem}
\begin{proof}
	The statement (a) is essential taken from  Theorem 2.1 in \cite{AJ}. The proof of (b) can be done similarly and hence we leave it to interested readers.
\end{proof}
\begin{proof}
	[Proof of Theorem \ref{mainthm2-above convex Lipschitz domains}:] The  $L^2$-boundedness of $\nabla^2 \Delta_D^{-1}$ and $\nabla^2 \Delta_N^{-1}$ was proved in \cite{A2} and \cite{AJ}, respectively. The Gaussian upper bound \eqref{GU-Neumann} follows from Theorem 3.2.9 in \cite{Da}. The H\"older continuity conditions \eqref{Holder-Dirichlet} and \eqref{Holder-Neumann} for any $\gamma\in (0,1)$ can be found in \cite{AR}. We now verify conditions \eqref{L2 gradient estimate} and \eqref{L2 gradient estimate-Neumann}. Let us take care of \eqref{L2 gradient estimate} first. By using (b) in Lemma \ref{lemAJ} and arguing similarly to the proof of Proposition 4.16 in \cite{DHMMY} we conclude that there exist a constant $\beta>0$ and $C>0$ so that
	\begin{equation}\label{L^2 strong estimates}
	\Big(\int_{x\in \Om} |\nabla_x^2 p_{t,\Delta_D}(x,y)|^2e^{\beta \f{|x-y|^2}{t}}dx\Big)^{1/2} \le C_\beta t^{-2}|B_\Om(y,\sqrt{t})|^{-1}, \ \ \ \forall y\in \Om, t>0
	\end{equation}
    which proves \eqref{L2 gradient estimate}.
    
    Similarly, we obtain that \eqref{L2 gradient estimate-Neumann} is satisfied.
    
    Hence, Theorem \ref{mainthm1-bounded domain} follows directly from Theorem \ref{thm1-Dirichlet}, Theorem \ref{thm2-Dirichlet}, Theorem \ref{thm1-Neumann} and Theorem \ref{thm2-Neumann}.
\end{proof}

\begin{proof}
	[Proof of Theorem \ref{mainthm3-upper half-space}:] The boundedness of $\nabla^2 \Delta_D^{-1}$ and $\nabla^2 \Delta_N^{-1}$ on $L^p(\Om)$ for $1<p<\vc$ is classical. See for example \cite{CKS2}. The Gaussian upper bound \eqref{GU-Neumann}, the H\"older continuity conditions \eqref{Holder-Dirichlet} and \eqref{Holder-Neumann}, and  conditions \eqref{L2 gradient estimate} and \eqref{L2 gradient estimate-Neumann} follow directly from the explicit expression for the kernels of $p_{t,\Delta_D}(x,y)$ and $p_{t,\Delta_N}(x,y)$.
	
	Therefore, Theorem \ref{mainthm3-upper half-space} follows immediately from Theorem \ref{thm3-Dirichlet} and Theorem \ref{thm3-Neumann}.
\end{proof}
{\bf Acknowledgement.} X. T. Duong was supported by the Australian Research Council through the research grant ARC DP160100153.

\end{document}